%% file: paper.tex
\documentclass[a4paper,11pt]{article}

\usepackage[UKenglish]{babel}
\usepackage[margin=2.5cm]{geometry}
\usepackage{hyperref}
\usepackage{amsmath}
\usepackage[capitalise]{cleveref}
\usepackage{commath}
\usepackage{graphicx}
\usepackage[numbers]{natbib}
\usepackage{amssymb}
\usepackage{mathtools}
\usepackage{bm}
\usepackage{amsthm}
\usepackage{stmaryrd}
\usepackage{subcaption}
\usepackage{pgfplots}
\usepackage{gradient_indicator}
\pgfplotsset{compat=newest}
\usetikzlibrary{calc}
\usepackage{siunitx}

\usepackage{authblk}

\newtheorem{assumption}{Assumption}
\crefname{assumption}{Assumption}{Assumptions}
\newtheorem{lemma}{Lemma}
\newtheorem{definition}{Definition}
\newtheorem{theorem}{Theorem}

\newtheorem{remark}{Remark}

\newcommand{\normv}[1]{|\!|\!|#1|\!|\!|_v}
\newcommand{\normvp}[1]{|\!|\!|#1|\!|\!|_{v'}}
\newcommand{\normL}[3]{\norm{#2}_{L^{#1}(#3)}}
\newcommand{\normpbar}[1]{\norm{#1}_p}
\newcommand{\normp}[1]{|\!|\!|#1|\!|\!|_p}
\newcommand{\norml}[1]{\norm{#1}_{\ell^2}}
\newcommand{\normH}[3]{\norm{#2}_{H^{#1}(#3)}}
\newcommand{\seminormH}[3]{\envert{#2}_{H^{#1}(#3)}}
\newcommand{\sumK}{\sum_{K \in \mathcal{T}_h}}

\title{Design and analysis of an exactly divergence-free hybridized
discontinuous Galerkin method for incompressible flows on meshes with
quadrilateral cells}

\author[1]{Joseph P.~Dean\thanks{\url{jpd62@cam.ac.uk},
\url{https://orcid.org/0000-0001-7499-3373}}}
\author[2]{Sander Rhebergen\thanks{\url{srheberg@uwaterloo.ca}
\url{https://orcid.org/0000-0001-6036-0356}}}
\author[1]{Garth N.~Wells\thanks{\url{gnw20@cam.ac.uk},
\url{https://orcid.org/0000-0001-5291-7951}}}

\affil[1]{Department of Engineering, University of Cambridge, United Kingdom}
\affil[2]{Department of Applied Mathematics, University of Waterloo, Canada}

\date{}
\begin{document}
\maketitle
\begin{abstract}
  \noindent
  We generalise a hybridized discontinuous Galerkin method for
  incompressible flow problems to non-affine cells, showing that with a
  suitable element mapping the generalised method preserves a key
  invariance property that eludes most methods, namely that any
  irrotational component of the prescribed force is exactly balanced by
  the pressure gradient and does not affect the velocity field. This
  invariance property can be preserved in the discrete problem if the
  incompressibility constraint is satisfied in a sufficiently strong
  sense. We derive sufficient conditions to guarantee discretely
  divergence-free functions are exactly divergence-free and give
  examples of divergence-free finite elements on meshes with triangular,
  quadrilateral, tetrahedral, or hexahedral cells generated by a
  (possibly non-affine) map from their respective reference cells. In
  the case of quadrilateral cells, we prove an optimal error estimate
  for the velocity field that does not depend on the pressure
  approximation. Our analysis is supported by numerical results.
\end{abstract}
\section{Introduction}
\label{sec:introduction}

Recently, there has been significant interest in numerical methods for
incompressible flow problems that preserve a fundamental invariance
property of the Stokes and Navier-Stokes equations; it can be shown
using a Helmholtz decomposition that any modification to the
irrotational component of the applied force is exactly balanced by the
pressure gradient and does not affect the velocity field
\cite{Linke2014,John2017}. Most of the classical mixed methods for
incompressible flow, including the Taylor--Hood \cite{Hood1974}, MINI
\cite{Arnold1984}, and Crouzeix--Raviart \cite{Crouzeix1973} elements,
do not preserve this property. This is because the divergence-free
constraint is only enforced approximately, and therefore discretely
divergence-free vector fields are not $L^2$-orthogonal to irrotational
fields, resulting in poor momentum balance i.e.~irrotational forces can
drive spurious flow (see \cite{Linke2014,John2017} for details). These
methods are not \emph{pressure robust} in the sense that the pressure
field can pollute the velocity approximation; error estimates for the
velocity field depend on the approximation error of the pressure field
scaled by the inverse of the viscosity. Hence, if the viscosity is small
or if the pressure field is poorly approximated, the error in the
velocity field can be large. This was demonstrated for colliding flow in
a cross-shaped domain by \citet{Linke2009}. Pressure robustness can also
be important to compute accurate solutions to buoyancy-driven flows and
fluids in hydrostatic equilibrium \cite{John2017}. A comprehensive
review of the role the divergence constraint plays in discretisations of
incompressible flow problems can be found in~\citet{John2017}.

Amongst the approaches to preserve the invariance property and obtain a
pressure-robust method, one is to enforce the divergence-free constraint
exactly in the discrete problem~\cite{John2017}. There are also other
benefits to using schemes that provide divergence-free velocity fields.
For instance, if mass is not conserved exactly, finite element methods
for the incompressible Navier--Stokes equations are typically not energy
stable~\cite{Evans2011}. Moreover, a divergence-free velocity field can
help avoid spurious results and instabilities when solving transport
equations (such as in turbulence models \cite{Peters2019}). Furthermore,
for applications involving particle advection, such as particle-mesh
methods or particle tracing, an initially uniform particle distribution
does not remain uniform as time progresses unless the divergence
constraint is satisfied exactly; instead, the particles clump together
in a non-physical manner (see~\cite{Maljaars2019}).

In the case of conforming mixed methods, finite element exterior
calculus \cite{Arnold2006} has proved to be a valuable tool
\cite{John2017} for constructing divergence-free schemes. Several
methods can be viewed in this framework, including the Scott--Vogelius
element \cite{Scott1985,Scott1985b} on barycentric-refined meshes and
the Guzm\'an--Neilan element \cite{Guzman:2014a,Guzman:2014b}. However,
methods of this variety typically require either high-order polynomials,
enriched spaces, special meshes, or tend to be challenging to extend to
three dimensions \cite{John2017,Neilan2018} due to the smoothness of the
Stokes complex \cite{John2017}. In general, it is difficult to balance
stability and incompressibility when restricted to conforming function
spaces~\cite{Qin1994}.

To balance the demands of stability and incompressibility, one approach
is non-conforming methods. Discontinuous Galerkin (DG) methods using
only $L^2$-conforming approximation spaces combined with efficient
(element-wise) post-processing techniques can yield divergence-free
velocity approximations, see \cite{Cockburn2004}. Post-processing can be
avoided if the finite element velocity space is instead
$H(\rm{div})$-conforming, as detailed in
\citet{Cockburn2007,Cockburn2004} and~\citet{Wang2007a}. However, due to
the large number of degrees of freedom introduced, DG methods are
relatively expensive compared to conforming methods. This motivated the
development of hybrid DG (HDG) methods which, in the spirit of hybrid
mixed methods \cite{Boffi2013}, introduce functions defined only over
the facets of the mesh and use static condensation to eliminate the cell
degrees-of-freedom. This reduces the number of globally coupled degrees
of freedom considerably, especially for high-order schemes
\cite{Cockburn2009b}. An HDG method for the Stokes equations was
introduced by \citet{Cockburn2009a} and is based on a
vorticity-velocity-pressure formulation. \citet{Lehrenfeld2016} present
an HDG method for the incompressible Navier--Stokes equations using an
$H(\rm{div})$-conforming finite element space for the velocity field. In
their scheme, only the tangential component of the velocity is
`hybridized'. \citet{Rhebergen2018a,Rhebergen2017,Rhebergen2020} enforce
continuity of the normal component via hybridization with a simple
modification of the scheme in \cite{Labeur2012}. The facet pressure
space is chosen to ensure that discretely divergence-free functions are
$H(\rm{div})$-conforming, otherwise pressure robustness is lost
\cite{Rhebergen2020}. Other examples of HDG methods for the Stokes
equations include \cite{Carrero2005a,Cockburn2005c,Cockburn2005d}.

Most of the analysis of divergence-free HDG schemes in the literature is
limited to meshes made of affine cells, i.e.~flat-faced simplices,
parallelograms, or parallelepipeds. The method in \cite{Rhebergen2018a}
only yields a divergence-free velocity field when the mesh is made of
affine simplices. It would clearly be advantageous for complex, curved
geometries to preserve pressure robustness for higher-order geometric
maps, and to be able to use tensor product cells for boundary layers and
to apply fast integration techniques. Examples of other divergence-free
methods for non-affine cells include \citet{Neilan2020}, who analyse a
Scott--Vogelius isoparametric finite element in two dimensions and
\citet{Evans2011}, who exploits the smoothness of B-splines on
single-patch domains to discretise a Stokes complex, yielding an inf-sup
stable and divergence-free scheme. \citet{Neilan2018} introduce a
divergence-free finite element pair for the Stokes problem that is
stable and conforming on shape-regular meshes made of flat-faced
quadrilateral cells. Piecewise quadratic and piecewise constant
polynomials are used for the velocity and pressure spaces respectively,
and after post-processing, the approximate velocity and pressure fields
are both second-order accurate. Methods that use
$H(\mathrm{div})$-conforming elements for the velocity space, such as
\cite{Cockburn2004,Cockburn2007,Lehrenfeld2016,Schroeder2018}, can also
provide exactly divergence-free solutions on non-affine cells.

This work extends and generalises the HDG method from
\cite{Rhebergen2018a,Rhebergen2020} to non-affine and non-simplex cells
whilst preserving the invariance property. Detailed analysis focuses on
flat-faced quadrilateral cells, but numerical examples are presented for
a range of non-affine cell types. We focus on the Stokes equations, but
the results can easily be generalised to the Navier--Stokes equations.
The layout of this paper is as follows: In \cref{sec:hdg_scheme} the
proposed HDG scheme is detailed, followed in
\cref{sec:ensuring_div_free} by the conditions which, if satisfied,
ensure that discretely divergence-free functions are exactly
divergence-free. In \cref{sec:press_rob_error_est}, a pressure robust
error estimate is derived, and in \cref{sec:proof_of_ass_on_quads} the
conditions required to apply our pressure robust error estimate are
proved for a divergence-free quadrilateral element. We present numerical
results to support our theoretical analysis in \cref{sec:num_ex} and
draw conclusions in \cref{sec:conclusion}.

\section{Stokes flow and the hybrid discontinuous Galerkin method}
\label{sec:hdg_scheme}

Consider viscous incompressible fluid flow in a domain $\Omega \subset
\mathbb{R}^d$, with $d \in \cbr{2, 3}$. The velocity $u: \Omega \to
\mathbb{R}^d$ and pressure $p: \Omega \to \mathbb{R}$ satisfy the Stokes
problem:
\begin{subequations}
  \label{eq:stokesproblem}
  \begin{alignat}{2}
    \label{eq:strong_momentum}
    - \nu \Delta u + \nabla p &= f &&\textnormal{ in } \Omega, \\
    \label{eq:strong_cont}
    \nabla \cdot u &= 0 &&\textnormal{ in } \Omega, \\
    \label{eq:diri_bound}
    u &= 0 &&\textnormal{ on } \partial \Omega,
  \end{alignat}
\end{subequations}
where $f: \Omega \to \mathbb{R}^d$ is a prescribed force and $\nu \in
\mathbb{R}^+$ is the kinematic viscosity. The pressure is determined
only up to a constant, so we seek pressure fields that satisfy the
condition
\begin{equation}
\label{eq:pressure_zero_mean}
    \int_\Omega p \dif x = 0.
\end{equation}

\subsection{Definitions}

Let $K$ denote a triangular/quadrilateral cell for $d = 2$ or a
tetrahedral/hexahedral cell for $d=3$ with non-zero diameter~$h_K$. Each
cell is assumed to be generated from a reference cell $\hat{K}$ via a
$\mathcal{C}^1$-diffeomorphism $T_K: \hat{K} \to K$ obtained from a
Lagrange geometric reference element~\citep{Ern2004}. Let the boundary
of each cell be denoted by $\partial K$ with outward unit normal $n$.
The domain $\Omega$ is partitioned into a mesh $\mathcal{T}_h = \{K\}$,
and $h$ is used to denote the maximum cell diameter in the mesh. The
geometric interpolation of $\Omega$, denoted $\Omega_h$, does not, in
general, coincide exactly with~$\Omega$~\citep{Ern2004}. However, for
simplicity, we assume that $\Omega_h = \Omega$. Let a facet in the mesh
be denoted by $F$, the set of all facets by $\mathcal{F}_h$, and let
$\Gamma_h \coloneqq \bigcup_{F \in \mathcal{F}_h} F$.

We now define the following finite element function spaces:
\begin{subequations}
  \begin{align}
    V_h &\coloneqq \cbr{ v_h \in [L^2(\Omega_h)]^d; \; v_h|_K
          \in V_h(K) \; \forall K \in \mathcal{T}_h},
          \label{eq:V_h_definition}
    \\
    \bar{V}_h &\coloneqq \cbr{ \bar{v}_h \in [L^2(\Gamma_h)]^d;
                \; \bar{v}_h|_F \in \bar{V}_h(F) \; \forall F \in \mathcal{F}_h, \;
                \bar{v}_h = 0 \textnormal{ on } \partial \Omega_h},
                \label{eq:Vbar_h_definition}
    \\
    Q_h &\coloneqq \cbr{ q_h \in L^2(\Omega_h); \; q_h|_K \in Q_h(K) \;
          \forall K \in \mathcal{T}_h},
    \\
    \bar{Q}_h &\coloneqq \cbr{ \bar{q}_h \in L_0^2(\Gamma_h); \;
                \bar{q}_h|_F \in \bar{Q}_h(F) \; \forall F \in \mathcal{F}_h}.
                \label{eq:Qbar_h_definition}
  \end{align}
\end{subequations}
The local spaces $V_h(K)$, $\bar{V}_h(F)$, $Q_h(K)$, and $\bar{Q}_h(F)$
are generated from their respective counterparts $V_h(\hat{K})$,
$\bar{V}_h(\hat{F})$, $Q_h(\hat{K})$, and $\bar{Q}_h(\hat{F})$ on the
reference element $\hat{K}$ or reference facet $\hat{F}$ using a linear
bijective mapping (see \cite[Proposition~1.61]{Ern2004}). For instance,
the space $V_h(K)$ is generated by
\begin{equation}
  V_h(K) = \cbr{\psi^{-1}_K(\hat{v}); \hat{v} \in V_h(\hat{K})},
  \label{eq:generate_V_h_on_K_from_Khat}
\end{equation}
where
\begin{equation}
  \psi_K : V_h(K) \to V_h(\hat{K}).
\end{equation}
We leave the reference spaces and the mapping $\psi_K$ undefined for
now. The spaces $V_h$ and $Q_h$ contain functions that are discontinuous
between cells, and similarly $\bar{V}_h$ and $\bar{Q}_h$ contain
functions that are discontinuous between facets. For convenience, we
introduce $\bm{V}_h \coloneqq V_h \times \bar{V}_h$ and $\bm{Q}_h
\coloneqq Q_h \times \bar{Q}_h$.

We follow Section~2 of \cite{Rhebergen2020} and introduce the extended
function spaces
\begin{subequations}
  \begin{align}
    V(h) &\coloneqq V_h + \sbr{H_0^1(\Omega_h)}^d \cap \sbr{H^2(\Omega_h)}^d,
    \\
    \bar{V}(h) &\coloneqq \bar{V}_h + [H_0^{3/2}(\Gamma_h)]^d,
    \\
    Q(h) &\coloneqq Q_h + L^2_0(\Omega_h) \cap H^1(\Omega_h),
    \\
    \bar{Q}(h) &\coloneqq \bar{Q}_h + H_0^{1/2}(\Gamma_h),
  \end{align}
\end{subequations}
where $L^2_0(\Omega_h)$ denotes the space of functions in
$L^2(\Omega_h)$ with zero mean. Further, we define the norms
\begin{equation}
  \normv{\bm{v}}^2 \coloneqq \sumK \normL{2}{\nabla v}{K}^2 +
  \sumK \frac{\alpha}{h_K}
  \normL{2}{\bar{v} - v}{\partial K}^2
\end{equation}
and
\begin{equation}
  \normvp{\bm{v}}^2 \coloneqq \normv{\bm{v}}^2 + \sumK
  \frac{h_K}{\alpha} \normL{2}{\partial_n v}{\partial K}^2
  \label{eq:norm_vp_def}
\end{equation}
on $V(h) \times \bar{V}(h)$, the norm
\begin{equation}
  \normpbar{\bar{q}}^2 \coloneqq \sumK h_K
  \normL{2}{\bar{q}}{\partial K}^2
\end{equation}
on $\bar{Q}(h)$, and the norm
\begin{equation}
  \normp{\bm{q}}^2 \coloneqq \normL{2}{q}{\Omega_h}^2 + \normpbar{\bar{q}}^2
\end{equation}
on $Q(h) \times \bar{Q}(h)$.

\subsection{Discrete problem}
\label{sec:discrete_problem}

The discrete problem consists of the following: given $f \in
\sbr{L^2(\Omega_h)}^{d}$, find $(\bm{u}_h, \bm{p}_h) \in \bm{V}_h \times
\bm{Q}_h$ such that
\begin{subequations}
  \label{eq:hdg_discrete_problem}
  \begin{alignat}{2}
    \label{eq:hdg_discrete_momentum}
    a_h(\bm{u}_h, \bm{v}_h) + b_h(v_h, \bm{p}_h) &= f_h(v_h)
    \quad &&\forall \bm{v}_h \in \bm{V}_h,
    \\
    \label{eq:hdg_discrete_continuity}
    b_h(u_h, \bm{q}_h) &= 0 \quad &&\forall \bm{q}_h \in \bm{Q}_h,
  \end{alignat}
\end{subequations}
where the bilinear forms are given by
\begin{align} \label{eq:a_h}
  \begin{split}
    a_h(\bm{u}_h, \bm{v}_h) \coloneqq &\sumK \int_K \nu \nabla u_h : \nabla v_h \dif x
    \\
    &- \sumK \int_{\partial K} \nu \del{(u_h - \bar{u}_h) \cdot \partial_nv_h
    + \partial_nu_h \cdot (v_h - \bar{v}_h)} \dif s
    \\
    &+ \sumK \int_{\partial K} \nu \frac{\alpha}{h_K} (u_h - \bar{u}_h)
    \cdot (v_h - \bar{v}_h) \dif s
  \end{split}
\end{align}
and
\begin{equation}
    b_h(v_h, \bm{p}_h) \coloneqq - \sumK \int_K p_h
    \nabla \cdot v_h \dif x + \sumK \int_{\partial K} v_h
    \cdot n \bar{p}_h \dif s,
\end{equation}
where $\alpha > 0$ is a penalty parameter. The linear form is defined as
\begin{equation}
  f_h(v_h) \coloneqq \int_{\Omega_h} f \cdot v_h \dif x.
\end{equation}

The above discrete problem is derived by posing mass and momentum
balances cell-wise, subject to `boundary conditions' provided by the
facet fields. Equations for the facet fields are obtained by enforcing
(weak) continuity of numerical fluxes across facets~\citep{Labeur2012}.
Critically, the numerical fluxes depend only on quantities that are
either \emph{local} to a particular cell or that live on the facet
shared by the cell and its neighbour. The cell fields communicate only
via the facet fields, meaning static condensation can be used to
eliminate the cell unknowns from the global system of equations.

\section{The invariance property and incompressibility}
\label{sec:ensuring_div_free}

There exists a fundamental invariance property of incompressible flow;
the irrotational component of the prescribed force $f \in
[L^2(\Omega)]^d$ affects only the pressure field, not the velocity
field. If $(u, p)$ solves the Stokes problem
\labelcref{eq:stokesproblem}, adding $\nabla \psi$ to the force $f$,
where $\psi \in H^1(\Omega) \setminus \mathbb{R}$, changes the solution
to $(u, p + \psi)$ \citep{John2017}. \emph{Pressure robust} methods
preserve this property in the sense that if $(u_h, p_h)$ solves the
discrete problem, then the above modification to the applied force
changes the discrete solution to $(u_h, p_h + \Pi_{Q_h} \psi)$, where
$\Pi_{Q_h}$ is the $L^2$-projection onto $Q_h$ (see \cite[Lemma
4.6]{John2017}). In this sense, contributions to the prescribed force
that only affect the pressure in the continuous problem also only affect
the pressure in the discrete problem~\citep{John2017}. As shown by
\citet{John2017}, for pressure robust methods the \emph{a priori} error
estimate for the velocity field does not depend on the pressure
approximation, in contrast to methods where the velocity error estimate
depends on the error in the pressure scaled by the inverse of the
viscosity. Pressure-robust methods are desirable since the pressure
field cannot pollute the approximation of the velocity field.

Preservation of the invariance property and pressure robustness are
intimately related to the sense in which the incompressibility
constraint is satisfied in the discrete problem.
\Cref{eq:hdg_discrete_continuity} constrains the discrete velocity
solution to lie in the space of \emph{discretely} divergence-free
functions,
\begin{equation}
  X_h \coloneqq \cbr{ v_h \in V_h; \; b_h(v_h, \bm{q}_h) = 0 \quad \forall
  \bm{q}_h \in \bm{Q}_h}.
\end{equation}
It is possible to construct pressure-robust methods by ensuring that
discretely divergence-free functions are also \emph{weakly}
divergence-free, the meaning of which we now make precise.

\begin{definition}[Weakly divergence-free]
  \label{def:weakly_div_free}
  A function $v \in H(\rm{div}; \Omega)$ is weakly divergence-free
  (written in short as $\nabla \cdot v = 0$) if
  \begin{equation}
      \int_\Omega v \cdot \nabla \psi \dif x = 0
        \quad \forall \psi \in C_0^\infty(\Omega).
  \end{equation}
\end{definition}
\Cref{def:weakly_div_free} implies $v$ is divergence-free in the
\emph{classical} sense in regular regions \emph{and} has continuous
normal components at their interfaces \cite{Bossavit1998}. We denote the
space of weakly divergence-free functions by
\begin{equation}
  X \coloneqq \left\{ v \in H(\textnormal{div}; \Omega); \nabla \cdot v
  = 0 \right\},
\end{equation}
where $\nabla \cdot v = 0$ is understood in the sense of
\cref{def:weakly_div_free}. For numerous common finite elements for
incompressible flows, the space of discretely divergence-free functions
is not a subset of the space of weakly divergence-free functions. In
some cases, e.g.~Taylor--Hood elements, the pressure space is not rich
enough to enforce the divergence constraint exactly. In other cases,
e.g.~standard discontinuous Galerkin methods, discretely divergence-free
functions may have discontinuous normal components.

For the discrete problem in \cref{sec:discrete_problem}, sufficient
conditions for the inclusion $X_h \subset X$ to be satisfied are:
\begin{enumerate}
  \item the property $v_h|_F \cdot n \in \bar{Q}_h(F) \ \forall v_h \in
      V_h(K)$, which ensures elements of $X_h$ have continuous normal
      components across facets; and
  \item the nesting $\nabla \cdot V_{h}(K) \subseteq Q_{h}(K)$, which
  ensures that elements of $X_h$ are divergence-free on cell interiors.
\end{enumerate}
For simplex cells, Lagrange elements have these sufficient properties
provided that the geometric map is affine, see~\citep{Rhebergen2018a}.
However, this is not the case when the geometric map is non-affine or
for non-simplex Lagrange elements. Indeed, in both of these cases,
simple calculations and numerical experiments confirm that the inclusion
$X_h \subset X$ does not hold. A similar issue affects the conforming
isoparametric Scott--Vogelius element, see~\cite{Neilan2020}.

We now demonstrate that, if functions are transformed between the
reference and physical elements in an appropriate manner, it is possible
to state sufficient conditions for the inclusion $X_h \subset X$ to hold
in terms of the relationships between the \emph{reference} finite
element function spaces instead of the \emph{physical} finite element
function spaces, and therefore independently of the geometric map $T_K$.
Key to this is the contravariant Piola transform defined by
\cite[eq.~(1.81)]{Ern2004}
\begin{equation} \label{eq:piola_transform}
    \psi_K(v) \coloneqq \det(J_K)J_K^{-1}(v \circ T_K),
\end{equation}
where $J_K$ is the Jacobian matrix of $T_K$. For an illustration of how
$\psi_K$ transforms vector fields, we refer the reader
to~\cite{Rognes2010}. The contravariant Piola transform preserves
divergence in the sense that
\cite[p.~2430]{Arnold2005}
\begin{equation} \label{eq:piola_preverves_div}
    \int_K \nabla \cdot v q \dif x =
    \int_{\hat{K}} \hat{\nabla} \cdot \hat{v} \hat{q} \dif \hat{x}
\end{equation}
and normal components in the sense that \cite[p.~2430]{Arnold2005}
\begin{equation} \label{eq:piola_preserves_normals}
    \int_{\partial K} v \cdot n q \dif s =
    \int_{\partial \hat{K}} \hat{v} \cdot \hat{n} \hat{q} \dif \hat{s},
\end{equation}
where $\hat{v}: \hat{K} \to \mathbb{R}^d$, $v = \psi_K^{-1}(\hat{v}): K
\to \mathbb{R}^d$, $\hat{q}: \hat{K} \to \mathbb{R}$, $q = \hat{q} \circ
T_K^{-1}: K \to \mathbb{R}$, and $\hat{n}$ denotes the unit outward
normal to $\partial \hat{K}$. Using the contravariant Piola transform to
construct $V_h$, it is possible to write sufficient conditions to ensure
$X_h \subset X$ that are independent of the geometric map $T_K$; we show
this in the following lemmas.
\begin{lemma} \label{lem:X_h_subset_Hdiv}
  Let $V_h(\hat{K}) \subset H(\textnormal{div};\hat{K})$ and
  $\bar{Q}_h(\hat{F})$ be finite dimensional polynomial spaces chosen
  such that
  \begin{equation} \label{eq:hdiv_conformity_condition_Khat}
    \bar{Q}_h(\hat{F}) \supseteq \cbr{\hat{v}_h|_{\hat{F}} \cdot \hat{n};
    \hat{v}_h \in V_h(\hat{K})}.
  \end{equation}
  Also let $V_h(K)$ be generated by
  \cref{eq:generate_V_h_on_K_from_Khat} with $\psi_K$ taken as the
  contravariant Piola transform (\cref{eq:piola_transform}) and generate
  $\bar{Q}_h(F)$ similarly but using the composition $\psi_F(\bar{q}_h)
  \coloneqq \bar{q}_h \circ T_K|_{\hat{F}}$. Then, discretely
  divergence-free functions have continuous normal components (i.e.~$X_h
  \subset H(\textnormal{div};\Omega_h)$).
\end{lemma}
\begin{proof}
  Let $v_h \in X_h$, giving $b_h(v_h, \bm{q}_h) = 0 \; \forall \bm{q}_h
  \in \bm{Q}_h$ by definition. Set $\bm{q}_h = (0, \bar{q}_h)$ on a
  single interior facet $F$ and $\bm{q}_h = (0, 0)$ elsewhere. The
  result then follows from transforming to the reference element and
  using the fact that the contravariant Piola transform preserves normal
  traces.
\end{proof}

\begin{remark}
  Since we consider the Stokes problem subject to the homogeneous
  Dirichlet boundary condition given in \cref{eq:diri_bound}, provided
  that the requirements of \cref{lem:X_h_subset_Hdiv} hold, we in fact
  have that $X_h \subset H_0(\textnormal{div};\Omega_h)$, where
  \begin{equation}
    H_0(\textnormal{div};\Omega_h) \coloneqq \left\{
    v \in H(\textnormal{div};\Omega_h); \;
    v \cdot n|_{\partial \Omega_h} = 0
    \right\}.
  \end{equation}
  In other words, the normal component of any discretely divergence-free
  function is exactly zero on the domain boundary. This can be seen by
  considering a facet on the domain boundary and applying a similar
  argument to the proof of \cref{lem:X_h_subset_Hdiv}.
\end{remark}

\begin{lemma} \label{lem:X_h_subset_X}
  If the conditions of \cref{lem:X_h_subset_Hdiv} are satisfied and if
  $V_h(\hat{K})$ and $Q_h(\hat{K})$ are chosen such that
  \begin{equation} \label{eq:div_V_h_in_Q_h_on_Khat}
    \hat{\nabla} \cdot V_h(\hat{K}) \subseteq Q_h(\hat{K}),
  \end{equation}
  then discretely divergence-free functions are weakly divergence-free
  (i.e.~$X_h \subset X$).
\end{lemma}
\begin{proof}
  Let $v_h \in X_h$, thus $b_h(v_h, \bm{q}_h) = 0 \; \forall \bm{q}_h
  \in \bm{Q}_h$ by definition. Set $\bm{q}_h = (q_h, 0)$ on an element
  $K$ and $\bm{q}_h = (0, 0)$ elsewhere. The result then follows from
  transforming to the reference element and using the fact that the
  contravariant Piola transform preserves divergence.
\end{proof}

Using these results, we now give some examples of finite elements for
which discretely divergence-free functions are weakly divergence-free,
which will be referred to as ``divergence-free elements''. Let
$\mathbb{P}_k$ denote the space of polynomials of degree at most $k$. If
the mesh is made of (possibly curved) simplex cells, then choosing
\begin{equation*}
  V_h(\hat{K}) \coloneqq [\mathbb{P}_k(\hat{K})]^d, \quad
  \bar{V}_h(\hat{F}) \coloneqq [\mathbb{P}_k(\hat{F})]^d, \quad
  Q_h(\hat{K}) \coloneqq \mathbb{P}_{k-1}(\hat{K}), \quad
  \textnormal{and} \quad \bar{Q}_h(\hat{F}) \coloneqq \mathbb{P}_{k}(\hat{F})
\end{equation*}
implies that discretely divergence-free functions are exactly
divergence-free. Indeed, for any $\hat{v}_h \in
[\mathbb{P}_k(\hat{K})]^d$ we have $\hat{v}_h|_{\hat{F}} \cdot \hat{n}
\in \mathbb{P}_{k}(\hat{F})$ and $\nabla \cdot [\mathbb{P}_k(\hat{K})]^d
= \mathbb{P}_{k-1}(\hat{K})$; thus,
\cref{lem:X_h_subset_Hdiv,lem:X_h_subset_X} can be applied.

In the case of (possibly curved) quadrilateral and hexahedral cells,
following an approach similar to the divergence-conforming methods
\cite{Cockburn2007,Lehrenfeld2016} we choose
\begin{equation*}
  V_h(\hat{K}) \coloneqq \mathbb{RT}_k(\hat{K}), \quad
  \bar{V}_h(\hat{F}) \coloneqq [\mathbb{Q}_k(\hat{F})]^d, \quad
  Q_h(\hat{K}) \coloneqq \mathbb{Q}_{k}(\hat{K}), \quad
  \textnormal{and} \quad \bar{Q}_h(\hat{F}) \coloneqq
  \mathbb{Q}_{k}(\hat{F}),
\end{equation*}
where $\mathbb{RT}_k$ denotes the Raviart--Thomas \cite{Raviart1977}
space of index $k$, and $\mathbb{Q}_k$ denotes the space of polynomials
of degree at most $k$ in each variable. These spaces have the property
that for any $\hat{v}_h \in \mathbb{RT}_k(\hat{K})$, we have
$\hat{v}_h|_{\hat{F}} \cdot \hat{n} \in \mathbb{Q}_{k}(\hat{F})$ and
$\nabla \cdot \mathbb{RT}_k(\hat{K}) = \mathbb{Q}_k(\hat{K})$
\cite[p.~97]{Boffi2013}. Another divergence-free choice is to instead
take
\begin{equation*}
  V_h(\hat{K}) \coloneqq \mathbb{BDM}_{k}(\hat{K}) \quad
  \textnormal{and} \quad Q_h(\hat{K}) \coloneqq \mathbb{Q}_{k-1}(\hat{K})
\end{equation*}
with the same facet spaces, where $\mathbb{BDM}_{k}$ is the
Brezzi--Douglas--Marini \cite{Brezzi1985} space of index $k$.

For the seemingly natural choice of
\begin{equation*}
  V_h(\hat{K}) \coloneqq [\mathbb{Q}_k(\hat{K})]^d, \quad
  \bar{V}_h(\hat{F}) \coloneqq [\mathbb{Q}_k(\hat{F})]^d, \quad
  Q_h(\hat{K}) \coloneqq \mathbb{Q}_{k-1}(\hat{K}), \quad
  \textnormal{and} \quad \bar{Q}_h(\hat{F}) \coloneqq \mathbb{Q}_{k}(\hat{F}),
\end{equation*}
we have $\nabla \cdot [\mathbb{Q}_k(\hat{K})]^d \nsubseteq
\mathbb{Q}_{k-1}(\hat{K})$ and it is possible to find elements of $X_h$
that are not elements of $X$. As a simple example, take $k=1$ and
consider the function $v_h \coloneqq \psi_K^{-1}((\hat{x}_0 + c
\hat{x}_0 \hat{x}_1, 0)) \in V_h(K)$, where $c \coloneqq - |\hat{K}| /
\int_{\hat{K}} \hat{x}_1 \dif \hat{x}$ is a constant. It is
straightforward to verify that $v_h$ satisfies $\int_{K} q_h \nabla
\cdot v_h \dif x = 0$ for all $q_h \in Q_h(K)$ and thus is
\emph{discretely} divergence-free, however, it is clearly not
\emph{weakly} divergence-free.

\section{Pressure-robust error estimate}
\label{sec:press_rob_error_est}

In this section, we use standard arguments to derive a pressure-robust
error estimate for the method described in \cref{sec:discrete_problem}.
We begin by stating some assumptions to ensure that the discrete problem
is well-posed and that $\bm{V}_h$ and $\bm{Q}_h$ have suitable
approximation properties. The validity of these assumptions is discussed
in \cref{rem:validity_of_assumptions} below.

\begin{assumption}[Discretely divergence-free functions are exactly
    divergence-free] \label{ass:disc_div_free_exact_div_free}
  Assume that $X_h \subset X$ (see \cref{sec:ensuring_div_free}).
\end{assumption}

\begin{assumption}[Consistency] \label{ass:consistency}
  If $(u, p) \in ([H_0^1(\Omega)]^d \cap [H^2(\Omega)]^d) \times
  (L_0^2(\Omega) \cap H^1(\Omega))$ solves the Stokes problem
  \labelcref{eq:stokesproblem}, then letting $\bm{u} \coloneqq (u,
  u|_{\Gamma_h})$ and $\bm{p} \coloneqq (p, p|_{\Gamma_h})$, it holds
  that
  \begin{subequations}
    \begin{alignat}{2}
      a_h(\bm{u}, \bm{v}_h) + b_h(v_h, \bm{p}) &= f_h(v_h) \quad && \forall
      \bm{v}_h \in \bm{V}_h,
      \\
      b_h(u, \bm{q}_h) &= 0 \quad &&\forall \bm{q}_h \in \bm{Q}_h.
    \end{alignat}
  \end{subequations}
\end{assumption}

\begin{assumption}[Norm equivalence on $\bm{V}_h$] \label{ass:norm_equiv}
  There exists a constant $c$, independent of $h$, such that for all
  $\bm{v}_h \in \bm{V}_h$
  \begin{equation}
      \normv{\bm{v}_h} \leq \normvp{\bm{v}_h} \leq c
      \normv{\bm{v}_h}.
  \end{equation}
  In other words, the norms $\normv{\cdot}$ and $\normvp{\cdot}$ are
  equivalent on $\bm{V}_h$.
\end{assumption}

\begin{assumption}[Stability of $a_h$] \label{ass:stability_of_a_h}
  There exists a $\beta_v > 0$ that is independent of $h$ and a constant
  $\alpha_0 > 0$ such that for $\alpha > \alpha_0$
  \begin{equation}
      a_h(\bm{v}_h, \bm{v}_h) \geq \nu \beta_v \normv{\bm{v}_h}^2
      \quad \forall \bm{v}_h \in \bm{V}_h.
  \end{equation}
\end{assumption}
\begin{assumption}[Boundedness of $a_h$] \label{ass:boundedness_of_a_h}
  There exists a constant $C_a > 0$, independent of $h$, such that
  $\forall \bm{u} \in V(h) \times \bar{V}(h)$ and $\forall \bm{v}_h \in
  \bm{V}_h$
  \begin{equation}
    |a_h(\bm{u}, \bm{v}_h)| \leq \nu C_a
    \normvp{\bm{u}} \normv{\bm{v}_h}.
  \end{equation}
\end{assumption}

\begin{assumption}[Stability of $b_h$] \label{ass:stability_of_b_h}
  There exists a constant $\beta_p > 0$, independent of $h$, such that
  \begin{equation}
    \inf_{\bm{q}_h \in \bm{Q}_h} \sup_{\bm{v}_h \in \bm{V}_h}
    \frac{b_h(v_h, \bm{q}_h)}{\normv{\bm{v}_h} \normp{\bm{q}_h}} \geq \beta_p.
  \end{equation}
\end{assumption}

\begin{assumption}[Boundedness of $b_h$] \label{ass:boundedness_of_b_h}
  There exists a constant $C_b > 0$ that is independent of $h$ such that
  $\forall \bm{v}_h \in \bm{V}_h$ and $\forall \bm{p} \in Q(h) \times
  \bar{Q}(h)$
  \begin{equation}
    |b_h(v_h, \bm{p})| \leq C_b \normv{\bm{v}_h} \normp{\bm{p}}.
  \end{equation}
\end{assumption}

\begin{assumption}[Approximation properties of $\bm{V}_h$ and $\bm{Q}_h$]
  \label{ass:approx_props}
  For any $(v, q) \in [H_0^{k+1}(\Omega)]^d \times H^k(\Omega)$ with $k
  \geq 1$, letting $\bm{v} \coloneqq (v, v|_{\Gamma_h})$ and $\bm{q}
  \coloneqq (q, q|_{\Gamma_h})$, there holds
  \begin{equation} \label{eq:vel_space_approximability}
    \inf_{\bm{v}_h \in \bm{V}_h} \normvp{\bm{v} - \bm{v}_h}
    \leq c h^k \normH{k+1}{v}{\Omega},
  \end{equation}
  and
  \begin{equation} \label{eq:pressure_space_approximability}
    \inf_{\bm{q}_h \in \bm{Q}_h} \normp{\bm{q} - \bm{q}_h} \leq
    c h^{k} \normH{k}{q}{\Omega},
  \end{equation}
  where $c$ is a constant independent of $h$.
\end{assumption}

We are now able to derive the following pressure robust error estimate.
\begin{theorem}[Error estimate] \label{th:error_estimate}
  Let $(u, p) \in [H_0^{k+1}(\Omega)]^d \times H^k(\Omega)$ solve the
  Stokes system given by
  \cref{eq:stokesproblem}
  with $k \geq 1$, and set $\bm{u} \coloneqq (u, u|_{\Gamma_h})$ and
  $\bm{p} \coloneqq (p, p|_{\Gamma_h})$. Let
  \crefrange{ass:disc_div_free_exact_div_free}{ass:approx_props} be
  satisfied and let $(\bm{u}_h, \bm{p}_h) \in \bm{V}_h \times \bm{Q}_h$
  solve the discrete problem given by \cref{eq:hdg_discrete_problem}.
  Then, we have
  \begin{equation}
      \normv{\bm{u} - \bm{u}_h} \leq c h^k \normH{k+1}{v}{\Omega},
  \end{equation}
  and
  \begin{equation}
      \normL{2}{p - p_h}{\Omega} \leq c h^k \left(\normH{k}{p}{\Omega}
      + \nu \normH{k+1}{u}{\Omega}\right),
  \end{equation}
  where $c > 0$ is independent of $h$. In addition, if some restrictions
  are placed on the shape of $\Omega$ \cite{Rhebergen2017} such that the
  regularity condition
  \begin{equation}
      \nu \normH{2}{u}{\Omega} + \normH{1}{p}{\Omega} \leq c_r
      \normL{2}{f}{\Omega},
  \end{equation}
  holds for some constant $c_r$, then
  \begin{equation}
      \normL{2}{u - u_h}{\Omega} \leq c h^{k+1} \normH{k+1}{u}{\Omega}.
  \end{equation}
\end{theorem}
\begin{proof}
  The velocity and pressure error estimates follow from
  \crefrange{ass:disc_div_free_exact_div_free}{ass:approx_props} using
  standard arguments, so we omit the details. The error estimate in the
  $L^2(\Omega)$-norm can be obtained using the Aubin--Nitsche duality
  method.
\end{proof}

\begin{remark}[Validity of assumptions] \label{rem:validity_of_assumptions}
  When the mesh cells are generated from an affine map from the
  reference simplex, square, or cube, e.g.~flat sided triangles,
  tetrahedra, parallelograms, and parallelepipeds, it is straightforward
  to verify that the assumptions hold for all of the divergence-free
  elements discussed in \cref{sec:ensuring_div_free} using an approach
  similar to \cite[Section~2]{Rhebergen2020}. We therefore omit the
  proofs in this work. When the maps are non-affine, it is more
  complicated to verify that the assumptions hold. The main difficulty
  is the loss of approximation accuracy of standard
  $H(\textrm{div})$-conforming elements on general meshes.
  \citet{Arnold2005} prove necessary and sufficient conditions for
  $H(\textrm{div})$-conforming elements to achieve optimal order
  approximation in $L^2$ on cells generated from a bilinear isomorphism
  of the square (i.e.~flat-faced quadrilaterals). These conditions are
  satisfied by Raviart-Thomas elements but not by
  Brezzi--Douglas--Marini elements. The order of approximation in the
  $L^2$-norm of $\mathbb{BDM}_k$ degrades from $k + 1$ in the affine
  case to $\lfloor(k + 1) / 2\rfloor$ in the bilinear case. The
  situation is more complicated in three dimensions;
  \citet{falk_gatto_monk_2011} show that standard
  $H(\textrm{div})$-conforming elements do not achieve optimal order
  approximation on cells generated by a trilinear isomorphism of the
  cube. In \cref{sec:proof_of_ass_on_quads}, we restrict the analysis to
  the divergence-free Raviart--Thomas element from
  \cref{sec:ensuring_div_free} on flat-faced quadrilateral cells,
  leaving more complex cell types and geometric maps to be investigated
  as further work.
\end{remark}

\section{Proof of assumptions for flat-faced quadrilateral cells}
\label{sec:proof_of_ass_on_quads}

Up to this point, the presentation has been quite general;
\cref{lem:X_h_subset_Hdiv,lem:X_h_subset_X} give sufficient conditions
for discretely divergence-free functions to be exactly divergence-free
regardless of whether the geometric map $T_K$ is affine or not, we have
given examples of divergence-free elements on simplices, quadrilateral,
and hexahedral cells in \cref{sec:ensuring_div_free}, and the pressure
robust error estimate derived in \cref{sec:press_rob_error_est} applies
to any divergence-free element that satisfies
\crefrange{ass:consistency}{ass:approx_props}. We now focus specifically
on the divergence-free Raviart--Thomas element discussed in
\cref{sec:ensuring_div_free} on families of shape-regular meshes of
flat-faced quadrilateral cells, for which we prove
\crefrange{ass:consistency}{ass:approx_props}. Recall that
\cref{ass:disc_div_free_exact_div_free} was proved for the
Raviart--Thomas element in \cref{sec:ensuring_div_free}.

Let the mesh be made of flat-faced quadrilateral cells which are the
image of the reference square under a bilinear diffeomorphism, i.e.~$T_K
\in [\mathbb{Q}_1(\hat{K})]^2$ \cite{Ern2021}. Recall that $J_K$ is the
Jacobian matrix of $T_K$. We assume each cell is convex, which implies
that \cite[Lemma~1.119]{Ern2004}
\begin{align} \label{eq:jacobian_geometric_bounds}
  \begin{split}
    \normL{\infty}{\det(J_K)}{\hat{K}} &\leq c h_K^2, \quad
    \normL{\infty}{\det(J_K^{-1})}{K} \leq c \frac{1}{\rho_K^2}, \quad \\
    \normL{\infty}{J_K}{\hat{K}} &\leq c h_K, \quad
    \normL{\infty}{J_K^{-1}}{K} \leq c \frac{h_K}{\rho_K^2},
  \end{split}
\end{align}
where $\rho_K \coloneqq \min_{i \in \{1:4\}}\rho_i$, $\rho_i$ is the
diameter of the largest circle that can be inscribed into the triangle
formed by the three vertices $(a_j)_{j \neq i}$ of $K$. Let the family
of meshes $\left\{\mathcal{T}_h\right\}_{h > 0}$ be shape-regular in the
sense that there exists a $\sigma_0$ such that for all $h$ and for all
$K \in \mathcal{T}_h$, $\sigma_K \coloneqq h_K / \rho_K \leq \sigma_0$
\cite{Ern2004}. As such, the analysis does not hold for, for example,
highly stretched (high aspect ratio) cells in the mesh. We refer to
\cite{Georgoulis2003} for more on the theoretical aspects of such
anisotropic cells.

\subsection{Some useful results on quadrilateral cells}

To prove \crefrange{ass:consistency}{ass:approx_props}, we require the
following results.

\begin{lemma}[Broken Poincar\'{e} inequality] \label{lem:broken_poincare_ineq}
  For all $\bm{v}_h \in \bm{V}_h$, there is a constant $c > 0$,
  independent of $h$, such that
  \begin{equation}
      \normL{2}{v_h}{\Omega} \leq c \normv{\bm{v}_h}.
  \end{equation}
\end{lemma}
\begin{proof}
  A simple consequence of \cite[Eq. (4.21)]{DiPietro2012}.
\end{proof}

\begin{lemma}[Inverse inequality] \label{lem:inverse_ineq}
  For all $K \in \mathcal{T}_h$ and $l \in \{1, 2\}$, there exists a
  constant $c > 0$, independent of $h$, such that $\forall v_h \in V_h$
  \begin{equation}
    \normH{l}{v_h}{K} \leq c h_K^{-1} \normH{l-1}{v_h}{K}.
  \end{equation}
\end{lemma}
\begin{proof}
  The result is stated for flat-faced quadrilateral cells in Lemma~3.6
  of~\cite{Wheeler2012} and proved for hexahedral cells in Lemma~3.5
  of~\cite{Ingram2010}.
\end{proof}

\begin{lemma}[Discrete Trace Inequality] \label{lem:discrete_trace_ineq}
  For all $K \in \mathcal{T}_h$, there exists a constant $c > 0$,
  independent of $h$, such that $\forall v_h \in V_h$
  \begin{equation} \label{eq:trace_ineq}
    \normL{2}{v_h}{\partial K} \leq c h_K^{-\frac{1}{2}}\normL{2}{v_h}{K}
  \end{equation}
  and
  \begin{equation} \label{eq:trace_ineq_normal_deriv}
    \normL{2}{\partial_n v_h}{\partial K} \leq c h_K^{-\frac{1}{2}}
    \normH{1}{v_h}{K}.
  \end{equation}
\end{lemma}
\begin{proof}
  The results follow from the continuous trace inequality
  \cite[Theorem~3.2]{Evans2013a}
  \begin{equation} \label{eq:cont_trace_ineq}
      \normL{2}{v}{\partial K}^2 \leq c \left(h_K^{-1} \normL{2}{v}{K}^2 +
      h_K \seminormH{1}{v}{K}^2\right),
  \end{equation}
  which holds for all $v \in H^1(K)$, and \cref{lem:inverse_ineq} (see
  \cite[Lemma~5.6.2]{Evans2011}).
\end{proof}

We now introduce the space
\begin{equation}
  Y_h \coloneqq \left\{v_h \in V_h; \; \llbracket v_h \cdot n \rrbracket = 0
  \textnormal{ on } F \; \forall F \in \mathcal{F}_h^i \right\},
\end{equation}
where $\mathcal{F}_h^i$ denotes the set of interior facets of the mesh
and $\llbracket v_h \cdot n \rrbracket \coloneqq v|_{K_0} \cdot n|_{K_0}
+ v|_{K_1} \cdot n|_{K_1}$, where $K_0$ and $K_1$ are the cells sharing
$F \in \mathcal{F}_h^i$. $Y_h$ consists of functions in $V_h$ with
continuous normal components across element boundaries. Let
$\mathcal{I}_h^{RT}: [H_0^1(\Omega)]^2 \to Y_h$ be the global
Raviart--Thomas interpolation operator. We then have the following
lemma.

\begin{lemma}
  Let $v \in [H^l(K)]^d$ and $m \in \{0:1\}$. If (i) $l=1$ or (ii) $v$
  is divergence-free and $l \in \{1 : k + 1\}$, then there exists a
  constant $c$, independent of $h_K$, such that for all $K \in
  \mathcal{T}_h$
  \begin{equation} \label{eq:parametric_RT_interp_error}
    \seminormH{m}{v - \mathcal{I}_h^{RT} v}{K} \leq
    c h_K^{l - m} \normH{l}{v}{K}.
  \end{equation}
\end{lemma}
\begin{proof}
  The result follows from \cite[Eq.~(3.66) and Remark~3.5]{Brezzi1991}.
\end{proof}

For a flat-faced quadrilateral cell $K$, let
\begin{equation}
  \mathcal{R}_k(\partial K) \coloneqq
  \left\{\bar{q}_h; \; \bar{q}_h \in L^2(\partial K),
  \bar{q}_h|_{F} \in \mathbb{Q}_k(F) \; \forall F \in
  \mathcal{F}_K \right\},
\end{equation}
where $\mathcal{F}_K$ denotes the set of facets belonging to $K$. We
then have the following.
\begin{lemma}[Lifting operator] \label{lem:lifting_op}
    For flat-faced quadrilateral cells, there exists a linear lifting
    operator $L: \mathcal{R}_k(\partial K) \to V_h(K)$, such that
    \begin{equation}
        (L \bar{q}_h) \cdot n = \bar{q}_h \quad \textnormal{and} \quad
        \normL{2}{L \bar{q}_h}{K} \leq c h_K^{\frac{1}{2}}
        \normL{2}{\bar{q}_h}{\partial K}
    \end{equation}
    for all $\bar{q}_h \in \mathcal{R}_k(\partial K)$.
\end{lemma}
\begin{proof}
  The proof is a simple modification of Proposition 2.4 of
  \cite{Du2019}. We begin by defining the lifting of $\bar{q}_h \in
  \mathcal{R}_k(\partial K)$ as the unique function $v_h \coloneqq
  L\bar{q}_h \in V_h(K)$ satisfying
  \begin{equation}
    \int_{\hat{K}} \hat{v}_h \cdot \hat{w}_h \dif \hat{x} = 0
    \quad \forall \hat{w}_h \in
    \mathbb{Q}_{k-1, k}(\hat{K}) \times \mathbb{Q}_{k, k-1}(\hat{K})
  \end{equation}
  and
  \begin{equation} \label{eq:lifting_def_facets}
    \int_{\partial \hat{K}} \hat{v}_h \cdot \hat{n} \hat{r}_h \dif \hat{s}
    = \int_{\partial \hat{K}} \check{\bar{q}}_h \hat{r}_h \dif \hat{s}
    \quad \forall \hat{r}_h \in \mathcal{R}_k(\partial \hat{K}),
  \end{equation}
  where $\hat{v}_h \coloneqq \psi_K(v_h)$ and $\check{\bar{q}}_h
  \coloneqq |\det(J_K)| \norml{J_K^{-T} \hat{n}} \bar{q}_h \circ T_K$;
  the motivation for the latter definition will become evident shortly.
  Now, for any $\hat{r}_h \in \mathcal{R}_k(\partial \hat{K})$ the
  function $r_h \coloneqq \hat{r}_h \circ T_K^{-1}$ is an element of
  $\mathcal{R}_k(\partial K)$ because, for flat-faced quadrilaterals,
  the geometric mapping $T_K$ is affine when restricted to the cell's
  facets \cite[p.~172]{Wheeler2012}. In addition, we have
  \begin{align}
    \begin{split}
      \int_{\partial K} v_h \cdot n r_h \dif s =
      \int_{\partial \hat{K}} \hat{v}_h \cdot \hat{n} \hat{r}_h \dif \hat{s}
      = \int_{\partial \hat{K}} \check{\bar{q}}_h \hat{r}_h \dif \hat{s} 
      = \int_{\partial K} \bar{q}_h r_h \dif s,
    \end{split}
  \end{align}
  where the first equality follows from
  \cref{eq:piola_preserves_normals}, the second from
  \cref{eq:lifting_def_facets}, and the third from the fact that the
  measures $\dif s$ and $\dif \hat{s}$ are related by
  \cite[eq.~(9.12a)]{Ern2021}
  \begin{equation} \label{eq:transformed_measures}
    \dif s = |\det(J_{K})| \norml{J_{K}^{-T}\hat{n}} \dif \hat{s}
  \end{equation}
  (this was the motivation behind the definition of
  $\check{\bar{q}}_h$). Since this result holds for all $r_h \in
  \mathcal{R}_k(\partial K)$ and $v_h \cdot n \in \mathcal{R}_k(\partial
  K)$ (due to \cite[eq.~(2.4.4)]{Boffi2013} and the fact that $T_K$ is
  affine on facets), we have that $v_h \cdot n = \bar{q}_h$. This proves
  the first result of the lemma.

  The bound on $v_h$ follows from standard scaling arguments and the
  fact that all norms on a finite dimensional space are equivalent.
\end{proof}

\begin{lemma}
  \label{lem:gen_quad_v_inf_bound}
  Let $\left\{\mathcal{T}_h\right\}_{h > 0}$ be a family of meshes made
  of convex quadrilateral cells with $V_h(\hat{K}) \coloneqq
  \mathbb{RT}_k(\hat{K})$ and $\bar{V}_h(\hat{F}) \coloneqq
  [\mathbb{Q}_k(\hat{F})]^d$ where $k \geq 1$. Also let $v \in
  [H^{k+1}(\Omega)]^2$. Then, setting $\bm{v}  \coloneqq (v,
  v|_{\Gamma_h})$, we have that
  \begin{equation}
    \inf_{\bm{v}_h \in \bm{V}_h} \normvp{\bm{v} - \bm{v}_h}
    \leq c h^k \normH{k+1}{v}{\Omega}.
  \end{equation}
\end{lemma}
\begin{proof}
  Let $\mathbb{S}_k(\hat{K})$ be a finite dimensional space spanned by
  the vectors
  \begin{equation}
    \begin{pmatrix}
      \hat{x}_1^i \hat{x}_2^j \\
      0
    \end{pmatrix}, \quad
    \begin{pmatrix}
      0 \\
      \hat{x}_1^j \hat{x}_2^i
    \end{pmatrix},
    \textnormal{ and }
    \begin{pmatrix}
        \hat{x}_1^{k + 1} \hat{x}_2^k \\
        - \hat{x}_1^k \hat{x}_2^{k + 1}
    \end{pmatrix},
  \end{equation}
  where $0 \leq i \leq k$ and $0 \leq j \leq k$. The space
  $\mathbb{S}_k(\hat{K})$ is a subspace of $\mathbb{RT}_k(\hat{K})$
  \cite[p.~2432]{Arnold2005} and thus, by definition, is a subspace of
  $V_h(\hat{K})$. Recall that $V_h(K)$ is defined via
  \cref{eq:generate_V_h_on_K_from_Khat} with $\psi_K$ being the
  contravariant Piola transform. By assumption, each flat-faced
  quadrilateral in the mesh is convex and can be generated from the
  reference square using a bilinear isomorphism, and thus Lemma 4.3 of
  \cite{Arnold2005} ensures that $V_h(K) \supseteq [\mathbb{P}_k(K)]^2$.
  Shape regularity allows Proposition~4.1.9 and Lemma~4.3.8 in
  \cite{Brenner2008} to be applied, which state that there exists a
  polynomial $v_h$ of degree less than $k + 1$ (and thus an element of
  $V_h(K)$) such that, for $0 \leq m \leq k + 1$ we have
  \begin{equation} \label{eq:bramble_hilbert_brenner2008}
    \seminormH{m}{v - v_h}{K} \leq c h_K^{k + 1 - m}
    \normH{k + 1}{v}{K}.
  \end{equation}
  In addition, let $\mathcal{I}_h^L: [H_0^{k + 1}(\Omega)]^2 \to
  \Theta_h$ be the Lagrange interpolation operator, where
  \begin{equation}
    \Theta_h \coloneqq \left\{\theta_h \in [\mathcal{C}^0(\bar{\Omega})]^2;
    \theta_h|_K \circ T_K \in [\mathbb{Q}_k(\hat{K})]^2 \;
    \forall K \in \mathcal{T}_h; \theta_h|_{\partial \Omega} = 0 \right\}.
  \end{equation}
  On shape-regular, flat-faced quadrilaterals we have
  \cite[eq.~(13.28)]{Ern2021}
  \begin{equation} \label{eq:lagrange_interp_error_est}
    \seminormH{m}{v - \mathcal{I}_h^L v}{K} \leq
    c h_K^{k + 1 - m} \normH{k + 1}{v}{K}.
  \end{equation}
  The desired result can then be obtained by bounding $\normvp{(v - v_h,
  v|_{\Gamma_h} - \mathcal{I}_h^L v|_{\Gamma_h})}$ using
  \cref{eq:bramble_hilbert_brenner2008,eq:lagrange_interp_error_est,eq:cont_trace_ineq}.
\end{proof}

\begin{remark}
  We remark that the proof of \cref{eq:bramble_hilbert_brenner2008} for
  flat-faced quadrilateral cells relied on $V_h(\hat{K}) \supseteq
  \mathbb{S}_K(\hat{K})$, which would not have been the case if
  $V_h(\hat{K})$ was taken to be the Brezzi--Douglas--Marini finite
  element space (see \cite{Arnold2005}).
\end{remark}

\begin{lemma}
  \label{lem:gen_quad_q_inf_bound}
  Let $\left\{\mathcal{T}_h\right\}_{h > 0}$ be a family of meshes made
  of flat-faced quadrilateral cells with $Q_h(\hat{K}) \coloneqq
  \mathbb{Q}_{k}(\hat{K})$ and $\bar{Q}_h(\hat{F}) \coloneqq
  \mathbb{Q}_{k}(\hat{F})$ where $k \geq 1$. Let $q \in H^k(\Omega)$ and
  set $\bm{q}  \coloneqq (q, q|_{\Gamma_h})$. Then, there holds
  \begin{equation}
    \inf_{\bm{q}_h \in \bm{Q}_h} \normp{\bm{q} - \bm{q}_h} \leq
    c h^{k} \seminormH{k}{q}{\Omega}.
  \end{equation}
\end{lemma}
\begin{proof}
  Let $\mathcal{I}_h^B : H^k(\Omega) \to Q_h \cap C^0(\bar{\Omega})$ be
  the operator defined by equation (4.1) in \cite{Bernardi1998} with the
  properties (see \cite{Bernardi1998}, Theorems~4.1 and~4.2)
  \begin{subequations}
    \begin{align}
      \normL{2}{q - \mathcal{I}_h^B q}{K} &\leq
      c h_K^k \seminormH{k}{q}{\omega(K)},
      \\
      \seminormH{1}{q - \mathcal{I}_h^B q}{K} &\leq
      h_K^{k-1} \seminormH{k}{q}{\omega(K)},
    \end{align}
  \end{subequations}
  where $\omega(K)$ is the union of all cells that share at least a
  corner with $K$ \cite{Bernardi1998}. The desired result then follows
  from using these properties and \cref{eq:cont_trace_ineq} to bound
  $\normp{(q - \mathcal{I}_h^B q, q|_{\Gamma_h} - \mathcal{I}_h^B
  q|_{\Gamma_h})}$ from above.
\end{proof}

We now present some results that are helpful to prove inf-sup stability
of $b_h$. Following Section 2.4 of \cite{Rhebergen2020}, we begin by
defining
\begin{equation}
  b_1(v_h, q_h) \coloneqq - \sumK \int_K q_h \nabla \cdot v_h \dif x
  \quad \textnormal{and} \quad
    b_2(v_h, \bar{q}_h) \coloneqq \sumK \int_{\partial K} v_h \cdot n \bar{q}_h
    \dif s.
\end{equation}
Let us redefine $\mathcal{I}_h^B: [H^1(\Omega)]^2 \to \Theta_h$ to be
the regularization operator defined by Eq.~(4.11) in \cite{Bernardi1998}
with the properties (see \cite[Theorems~4.3 and~4.4]{Bernardi1998}):
\begin{equation}
  \normL{2}{v - \mathcal{I}_h^B v}{K} \leq
  c h_K \seminormH{1}{v}{\omega(K)}
  \label{eq:bernadi_interp_L2}
\end{equation}
and
\begin{equation}
  \normH{1}{v - \mathcal{I}_h^B v}{K} \leq
  c \seminormH{1}{v}{\omega(K)},
  \label{eq:bernadi_interp_H1}
\end{equation}
for $k\geq 1$, where we recall that $\omega(K)$ is the union of all
cells that share at least a corner with $K$. We then have the following
lemmas.

\begin{lemma}[Fortin operator] \label{lem:fortin}
  Let $k \geq 1$. Then, for any function $v \in [H_0^1(\Omega)]^2$,
  there holds
  \begin{equation}
    b_1(\mathcal{I}_h^{RT} v, q_h) = b_1(v, q_h) \quad \forall q_h \in Q_h
  \end{equation}
  and
  \begin{equation}
    \normv{(\mathcal{I}_h^{RT} v, \mathcal{I}_h^B v|_{\Gamma_h})} \leq c
    (1 + \alpha)^{\frac{1}{2}} \normH{1}{v}{\Omega}.
  \end{equation}
\end{lemma}
\begin{proof}
  The first result is given in the last paragraph of
  \cite[p.~2443]{Arnold2005} for meshes made of flat-faced quadrilateral
  elements. The bound on $\normv{(\mathcal{I}_h^{RT} v, \mathcal{I}_h^B
  v|_{\Gamma_h})}$ then follows from \cref{eq:bernadi_interp_L2},
  \cref{eq:bernadi_interp_H1}, \cref{eq:parametric_RT_interp_error}, the
  Cauchy--Schwarz inequality, and \cref{eq:cont_trace_ineq}.
\end{proof}

\begin{lemma}[Stability of $b_1$] \label{lem:stability_b_1}
  There exists a constant $\beta_1 > 0$, independent of $h$, such that
  \begin{equation}
    \inf_{q_h \in Q_h} \sup_{\bm{v}_h \in Y_h \times \bar{V}_h}
    \frac{b_1(v_h, q_h)}{\normv{\bm{v}_h} \normL{2}{q_h}{\Omega}}
    \geq \beta_1.
  \end{equation}
\end{lemma}
\begin{proof}
  Since the Hilbert complex
  \begin{equation}
      [H^1_0(\Omega)]^2 \xrightarrow{\nabla \cdot} L^2_0(\Omega)
  \end{equation}
  is bounded and exact at $L_0^2(\Omega)$, the divergence operator
  $\nabla \cdot: [H^1_0(\Omega)]^2 \to L^2_0(\Omega)$ is continuous and
  surjective \cite[Theorem 2.1.1]{Evans2011}. Hence, for all $q_h \in
  Q_h$, there exists a $v_{q_h} \in [H^1_0(\Omega)]^2$ such that $\nabla
  \cdot v_{q_h} = q_h$ and $\beta_c \normH{1}{v_{q_h}}{\Omega} \leq
  \normL{2}{q_h}{\Omega}$. Therefore, using \cref{lem:fortin}, we have
  that
  \begin{align}
    \begin{split}
      \inf_{q_h \in Q_h} \sup_{\bm{v}_h \in Y_h \times \bar{V}_h}
      \frac{b_1(v_h, q_h)}{\normv{\bm{v}_h}\normL{2}{q_h}{\Omega}}
      &\geq - \inf_{q_h \in Q_h} \frac{b_1(\mathcal{I}_h^{RT} v_{q_h}, q_h)}{\normv{(\mathcal{I}_h^{RT} v_{q_h},
      \mathcal{I}_h^B v_{q_h}|_{\Gamma_h})}\normL{2}{q_h}{\Omega}} \\
      &\geq - \inf_{q_h \in Q_h} \frac{b_1(v_{q_h}, q_h)}{c(1 + \alpha)^{\frac{1}{2}}
      \normH{1}{v_{q_h}}{\Omega} \normL{2}{q_h}{\Omega}} \\
      &\geq \frac{\beta_c}{c(1 + \alpha)^{\frac{1}{2}}}
      \inf_{q_h \in Q_h} \frac{\sumK \int_K q_h \nabla \cdot v_{q_h} \dif x}
      {\normL{2}{q_h}{\Omega}^2} \\
      &\geq \frac{\beta_c}{c(1 + \alpha)^{\frac{1}{2}}}.
    \end{split}
  \end{align}
  Thus, the lemma holds with $\beta_1 = \beta_c / c(1 +
  \alpha)^{\frac{1}{2}}$.
\end{proof}

\begin{lemma}[Stability of $b_2$] \label{lem:stability_b_2}
  There exists a constant $\beta_2 > 0$, independent of $h$, such that
  \begin{equation}
    \inf_{q_h \in \bar{Q}_h} \sup_{\bm{v}_h \in \bm{V}_h}
    \frac{b_2(v_h, \bar{q}_h)}{\normv{\bm{v}_h} \normpbar{\bar{q}_h}}
    \geq \beta_2.
  \end{equation}
\end{lemma}
\begin{proof}
  We follow an approach similar to the proof of Lemma~3 in
  \cite{Rhebergen2018b}. First, note that $\bar{Q}_h$ consists of
  piecewise discontinuous polynomials because the geometric mapping for
  a flat-faced quadrilateral cell is affine on each facet. Hence, for
  any $\bar{q}_h \in \bar{Q}_h$ and for any element $K \in
  \mathcal{T}_h$, we have that $\bar{q}_h|_F \in \mathbb{P}_k(F) \;
  \forall F \in \mathcal{F}_K$ (in two-dimensions $\mathbb{Q}_k(F) =
  \mathbb{P}_k(F)$ \cite[p.~98]{Boffi2013}). Thus, from
  \cref{lem:inverse_ineq,lem:discrete_trace_ineq,lem:lifting_op} we have
  that for all $\bar{q}_h \in \bar{Q}_h$ there holds
  \begin{equation}
      \normv{(h_KL\bar{q}_h,0)} \le c \normpbar{\bar{q}_h}.
  \end{equation}
  Hence, we have
  \begin{align}
    \begin{split}
      \inf_{q_h \in \bar{Q}_h} \sup_{\boldsymbol{v}_h \in \boldsymbol{V}_h}
      \frac{b_2(v_h, \bar{q}_h)}{\normv{\boldsymbol{v}_h} \normpbar{\bar{q}_h}}
      &\geq \inf_{q_h \in \bar{Q}_h}
        \frac{\sumK \int_{\partial K} (h_KL\bar{q}_h) \cdot n \bar{q}_h \dif s}
        {\normv{(h_KL\bar{q}_h, 0)}\normpbar{\bar{q}_h}}
      \\
      &\geq c \inf_{q_h \in \bar{Q}_h}
        \frac{\sumK h_K\int_{\partial K} \bar{q}_h^2 \dif s}
        {\normpbar{\bar{q}_h}^2}
      \\
      &= c.
    \end{split}
  \end{align}
\end{proof}

\begin{lemma}[Boundedness of $b_1$] \label{lem:bound_b_1}
  For all $\bm{v}_h \in \bm{V}_h$ and all $\bm{p} \in Q(h) \times
  \bar{Q}(h)$
  \begin{equation}
      |b_1(v_h, p)| \leq \normv{\bm{v}_h} \normp{\bm{p}}.
  \end{equation}
\end{lemma}
\begin{proof}
  A simple consequence of the Cauchy--Schwarz inequality.
\end{proof}

\begin{lemma}[Boundedness of $b_2$] \label{lem:bound_b_2}
  For all $\bm{v}_h \in \bm{V}_h$ and all $\bm{p} \in Q(h) \times
  \bar{Q}(h)$, there exists a constant $c > 0$, independent of $h$, such
  that
  \begin{equation}
      |b_2(v_h, \bar{p})| \leq c \normv{\bm{v}_h} \normp{\bm{p}}.
  \end{equation}
\end{lemma}
\begin{proof}
  A simple consequence of the Cauchy--Schwarz inequality and the fact
  that the facet functions are single-valued.
\end{proof}

We also require a reduced version of Theorem 3.1 from \cite{Howell2011},
which is stated below for convenience.

\begin{theorem} \label{th:split_stability}
  Let $U$, $P_1$, and $P_2$ be reflexive Banach spaces, and let $b_1: U
  \times P_1 \to \mathbb{R}$ and $b_2: U \times P_2 \to \mathbb{R}$ be
  bilinear and bounded. Also, let
  \begin{equation} \label{eq:z_b_2}
    Z_{b_2} \coloneqq \left\{v \in U; \; b_2(v, p_2) = 0 \;
    \; \forall p_2 \in P_2 \right\} \subset U.
  \end{equation}
  Then, the following are equivalent:
  \begin{enumerate}
    \item There exists a $c > 0$ such that
    \begin{equation}
        \sup_{v \in U} \frac{b_1(v, p_1) + b_2(v, p_2)}{||v||_U}
        \geq c \left(||p_1||_{P_1} + ||p_2||_{P_2}\right)
        \quad \forall (p_1, p_2) \in P_1 \times P_2
    \end{equation}
    \item There exists a $c > 0$ such that
    \begin{equation}
        \sup_{v \in Z_{b_2}} \frac{b_1(v, p_1)}{||v||_U}
        \geq c ||p_1||_{P_1} \; \forall p_1 \in P_1
        \quad \textnormal{and} \quad
        \sup_{v \in U} \frac{b(v, p_2)}{||v||_U}
        \geq c ||p_2||_{P_2} \; \forall p_2 \in P_2
    \end{equation}
  \end{enumerate}
\end{theorem}

\subsection{Proof of assumptions}

We are now able to prove \crefrange{ass:consistency}{ass:approx_props}
for meshes made of flat-faced quadrilateral cells. Recall that
\cref{ass:disc_div_free_exact_div_free} was proved in
\cref{sec:ensuring_div_free}.

\begin{proof}[\textbf{Proof of \cref{ass:consistency}}]
  The proof is identical to \cite[Lemma~4.1]{Rhebergen2017}, so we omit
  it here.
\end{proof}

\begin{proof}[\textbf{Proof of \cref{ass:norm_equiv}}]
  The result follows from $\bm{V}_h$ being finite dimensional,
  \cref{lem:discrete_trace_ineq,lem:broken_poincare_ineq}.
\end{proof}

\begin{proof} [\textbf{Proof of \cref{ass:stability_of_a_h}}]
  The result follows from a standard argument; consider $a_h(\bm{v}_h,
  \bm{v}_h)$, where $\bm{v}_h \in \bm{V}_h$, and apply the
  Cauchy--Schwarz inequality, Young's inequality,
  \cref{lem:discrete_trace_ineq}, and \cref{lem:broken_poincare_ineq},
  yielding
  \begin{equation}
    a_h(\bm{v}_h, \bm{v}_h) \geq \nu \left(\frac{1}{2} - \frac{c}{\alpha}\right)
              \normv{\bm{v}_h}^2,
  \end{equation}
  where $c$ depends on the constants in
  \cref{lem:discrete_trace_ineq,lem:broken_poincare_ineq}. Therefore,
  provided the penalty parameter is chosen such that $\alpha > 2 c$ we
  have that \cref{ass:stability_of_a_h} holds with $\beta_v =
  \frac{1}{2} - \frac{c}{\alpha} > 0$.
\end{proof}

\begin{proof}[\textbf{Proof of \cref{ass:boundedness_of_a_h}}]
  A simple consequence of the Cauchy--Schwarz inequality,
  \cref{lem:discrete_trace_ineq}, and \cref{lem:broken_poincare_ineq},
  yielding
  \begin{equation}
      |a_h(\bm{u}, \bm{v}_h)| \leq \nu c(1 + \alpha^{-\frac{1}{2}})
      \normvp{\bm{u}} \normv{\bm{v}_h},
  \end{equation}
  where $c$ depends on the constants in
  \cref{lem:discrete_trace_ineq,lem:broken_poincare_ineq}.
\end{proof}

\begin{proof}[\textbf{Proof of \cref{ass:stability_of_b_h}}]
  The proof of the assumption is almost identical to the proof of Lemma
  8 in \cite{Rhebergen2020}. Apply \cref{th:split_stability} with $U
  \coloneqq \bm{V}_h$, $P_1 \coloneqq Q_h$, $P_2 \coloneqq \bar{Q}_h$,
  and $Z_{b_2} \coloneqq Y_h \times \bar{V}_h$. \Cref{eq:z_b_2} is
  satisfied because every function in $Y_h$ has continuous normal
  components across facets. \Cref{lem:stability_b_1,lem:stability_b_2}
  then imply that \cref{ass:stability_of_b_h} holds by equivalence of
  items 1 and 2 in \cref{th:split_stability}.
\end{proof}

\begin{proof}[\textbf{Proof of \cref{ass:boundedness_of_b_h}}]
  The result follows trivially from \cref{lem:bound_b_1,lem:bound_b_2}.
\end{proof}

\begin{proof}[\textbf{Proof of \cref{ass:approx_props}}]
  The result is an immediate consequence of
  \cref{lem:gen_quad_v_inf_bound,lem:gen_quad_q_inf_bound}.
\end{proof}

\section{Numerical examples}
\label{sec:num_ex}

We present some numerical results to support our theoretical analysis.
We take $\alpha$ to be $16 k^2$ and consider meshes with quadrilateral
cells. We use the Raviart--Thomas based element discussed in
\cref{sec:ensuring_div_free,sec:proof_of_ass_on_quads} since it
satisfies the assumptions required to apply \cref{th:error_estimate} on
these meshes. FEniCSx
\cite{scroggs:2022,scroggs:2022b,fenics:book,dolfinx_web} was used for
the first two examples, and NGSolve \cite{Schoberl2014} was used for the
final example to demonstrate mixed topology meshes. Boundary data is
interpolated into $\bar{V}_h$, and $f$ is evaluated at quadrature
points. In the FEniCSx examples, the linear system of equations is
solved using the MUMPS \cite{MUMPS:1,MUMPS:2} sparse direct solver via
PETSc \cite{petsc-web-page,petsc-user-ref,petsc-efficient} and we set
the appropriate flags to handle the nullspace of constants. In the
NGSolve example, we use UMFPACK \cite{umfpack}. The code for the
numerical examples is available in \cite{code}.

\subsection{Stokes flow in a square domain}

Let $\Omega \coloneqq (0, 1)^2$ and choose the data such that the exact
solution to the Stokes problem is given by
\begin{equation}
    u(x) =
    \begin{pmatrix}
        \sin(\pi x_1) \sin(\pi x_2) \\
        \cos(\pi x_1) \cos(\pi x_2)
    \end{pmatrix}
    \quad \textnormal{and} \quad
    p(x) = \sin(\pi x_1) \cos(\pi x_2) + c_p,
\end{equation}
where $c_p$ is a constant such that the mean pressure is zero.

To investigate the convergence of the method, a family of meshes made of
trapezoidal cells are used. Each trapezium in every mesh is similar, so
the geometric mappings do not tend to affine as the mesh is refined.
\Cref{fig:square2_sol} shows the computed solution with $k=2$ and $\nu =
1$ for one of the meshes in the family.
\begin{figure}
    \centering
    \begin{subfigure}[b]{0.4\textwidth}
        \centering
        \includegraphics[width=\textwidth]{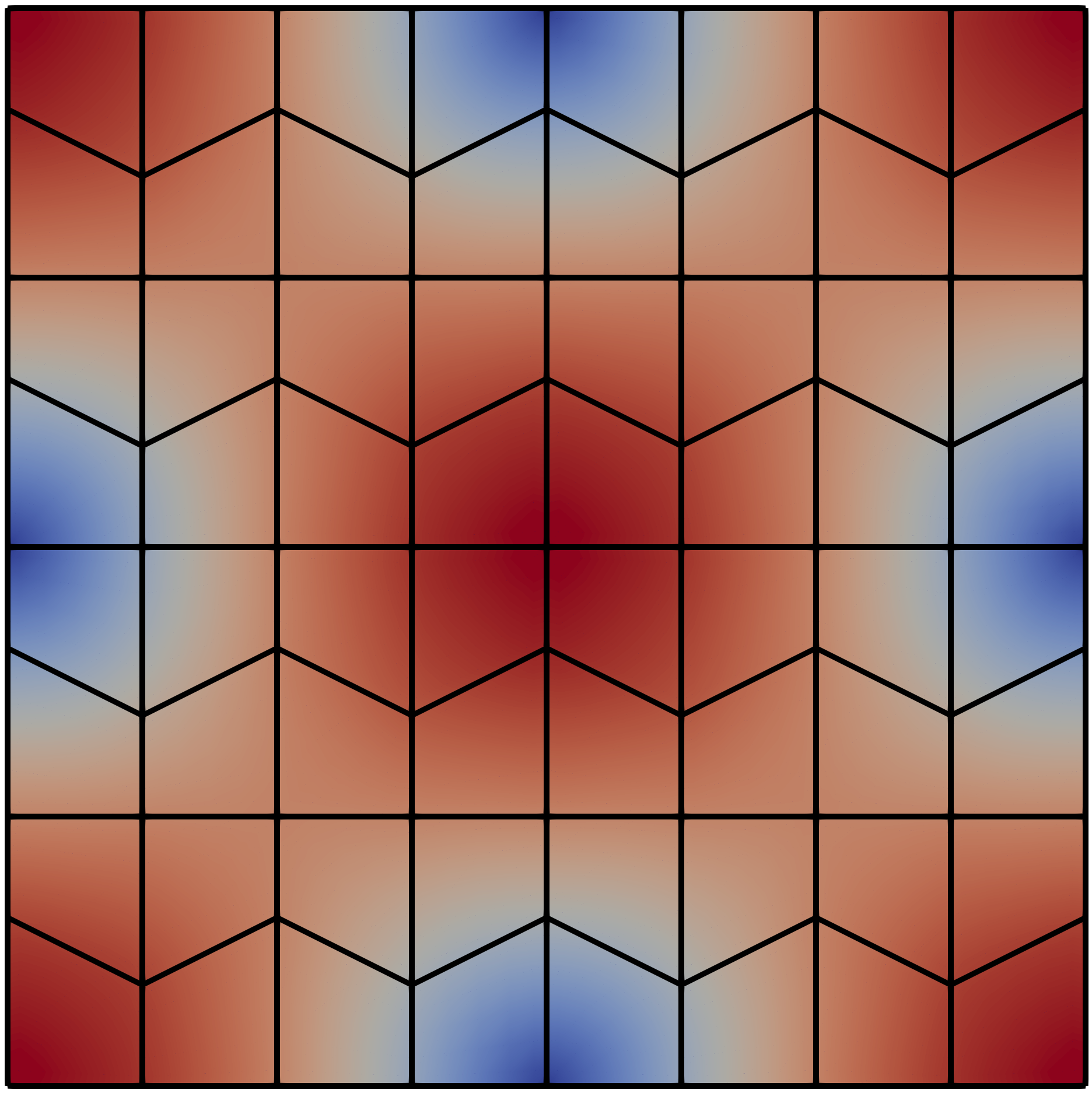}
        \caption{Velocity magnitude}
    \end{subfigure}
    \hfill
    \begin{subfigure}[b]{0.4\textwidth}
        \centering
        \includegraphics[width=\textwidth]{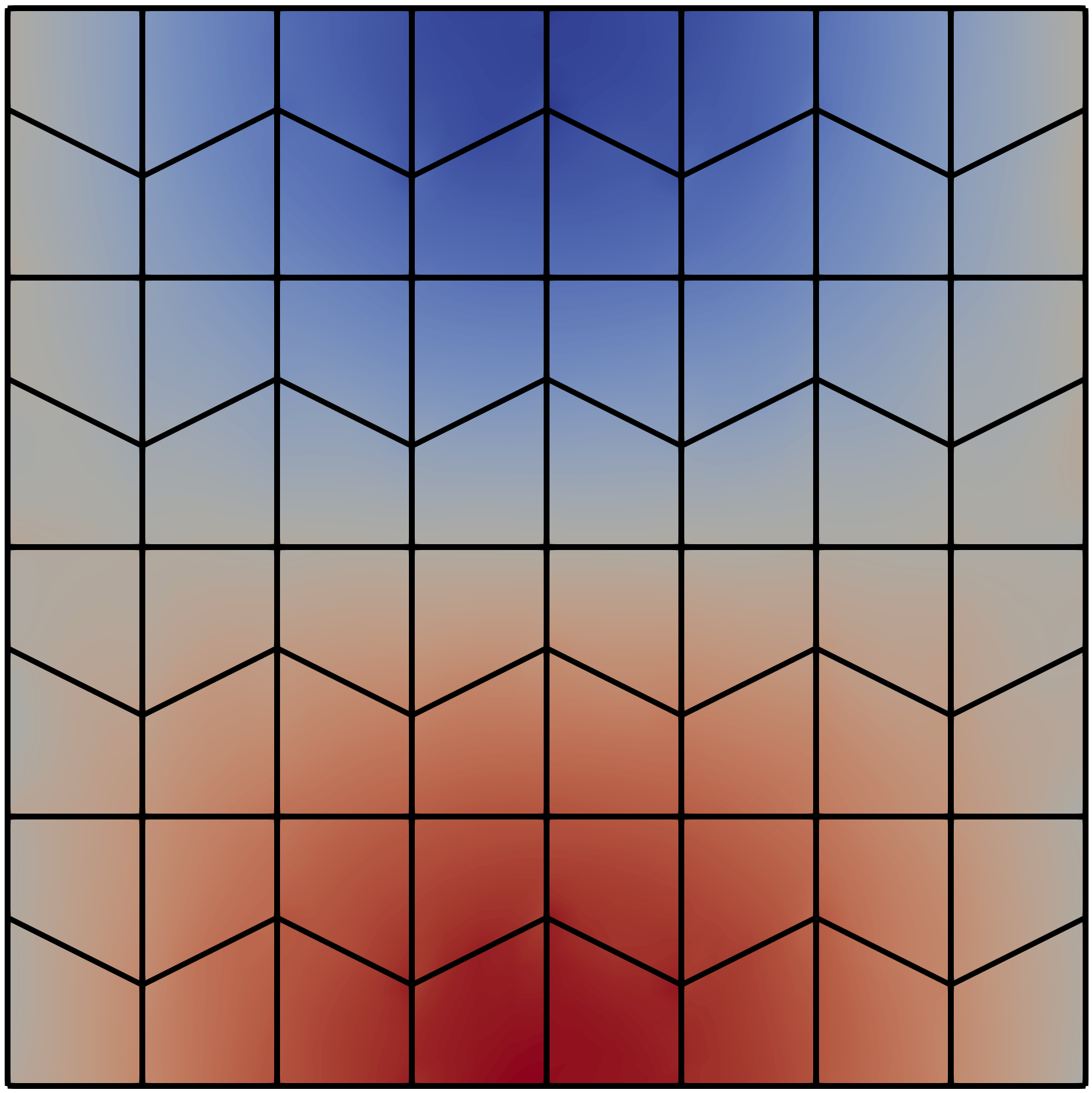}
        \caption{Pressure}
    \end{subfigure}
    \caption{Computed solution to a simple Stokes problem on a mesh of
    trapezium shaped elements.}
    \label{fig:square2_sol}
\end{figure}
\Cref{fig:square_2_conv_u_p} shows the $L^2(\Omega)$-norm of the error
in the velocity field, $e_u$, and pressure field, $e_p$, as a function
of $h$ for different values of $k$.
\begin{figure}
    \centering
    \begin{subfigure}[b]{0.49\textwidth}
        \centering
        \resizebox{\textwidth}{!}{
        \input{square2_conv_l2_errors_u.tex}}
        \caption{Velocity field}
    \end{subfigure}
    \hfill
    \begin{subfigure}[b]{0.49\textwidth}
        \centering
        \resizebox{\textwidth}{!}{
        \input{square2_conv_l2_errors_p.tex}}
        \caption{Pressure field}
    \end{subfigure}
    \caption{The $L^2(\Omega)$-norm of the error in the velocity and
    pressure fields as a function $h$. The velocity and pressure
    converge at rates $k + 1$ and $k$, respectively.}
    \label{fig:square_2_conv_u_p}
\end{figure}
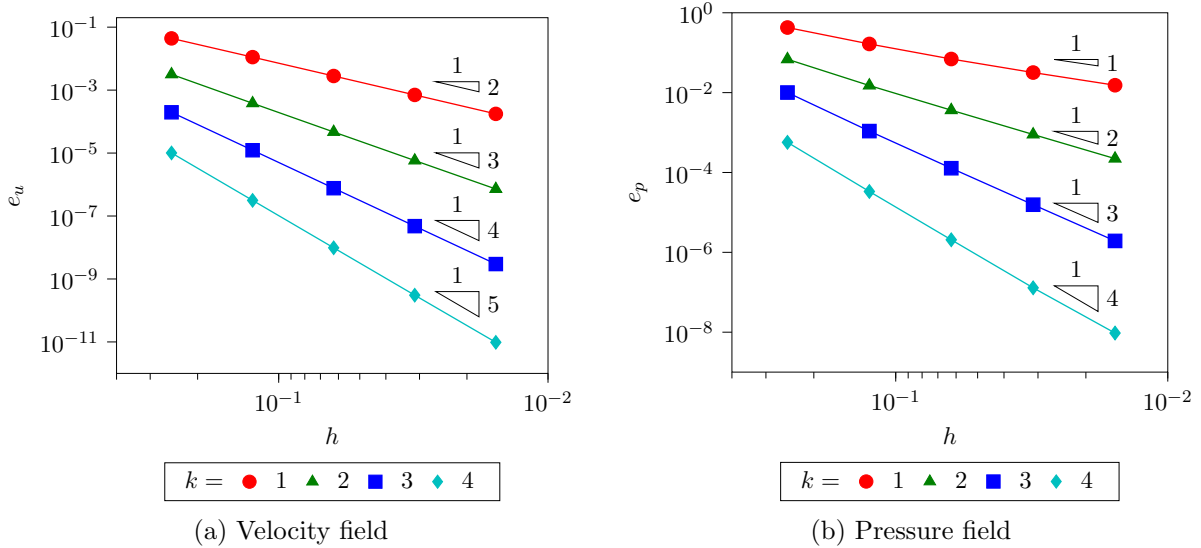
The observed rate of convergence is $k + 1$ for the velocity field and
$k$ for the pressure field, which is in agreement with the theory
presented in \cref{sec:press_rob_error_est}.

\Cref{fig:square_2_conv_div_jump} shows the $L^2(\Omega)$-norm of the
error in the divergence of the velocity field, $e_{\nabla \cdot u}$, and
the $L^2(\Gamma_h)$-norm of the jump in the normal component of the
velocity across element boundaries, $e_{\llbracket u \rrbracket}$, as a
function of $h$. We compare the present method with a naive extension of
the method from \citet{Rhebergen2018a} to quadrilateral cells, which
differs from the present method in that firstly, the finite element
spaces are taken to be
\begin{equation*}
    V_h(\hat{K}) \coloneqq [\mathbb{Q}_k(\hat{K})]^d, \quad
    \bar{V}_h(\hat{F}) \coloneqq [\mathbb{Q}_k(\hat{F})]^d, \quad
    Q_h(\hat{K}) \coloneqq \mathbb{Q}_{k-1}(\hat{K}), \quad
    \textnormal{and} \quad \bar{Q}_h(\hat{F}) \coloneqq \mathbb{Q}_{k}(\hat{F});
\end{equation*}
and secondly, composition is used to map from the reference to the
physical velocity space, rather than the contravariant Piola transform.
\begin{figure}
    \centering
    \begin{subfigure}[b]{0.49\textwidth}
        \centering
        \resizebox{\textwidth}{!}{
        \input{square2_conv_div_errors.tex}}
        \caption{Divergence error}
    \end{subfigure}
    \hfill
    \begin{subfigure}[b]{0.49\textwidth}
        \centering
        \resizebox{\textwidth}{!}{
        \input{square2_conv_jump_errors.tex}}
        \caption{Jump error}
    \end{subfigure}
    \caption{The divergence and jump errors as a function of $h$ for the
    present method (PM) and the method from \cite{Rhebergen2018a} (RW).
    Only the present method conserves mass exactly.}
    \label{fig:square_2_conv_div_jump}
\end{figure}
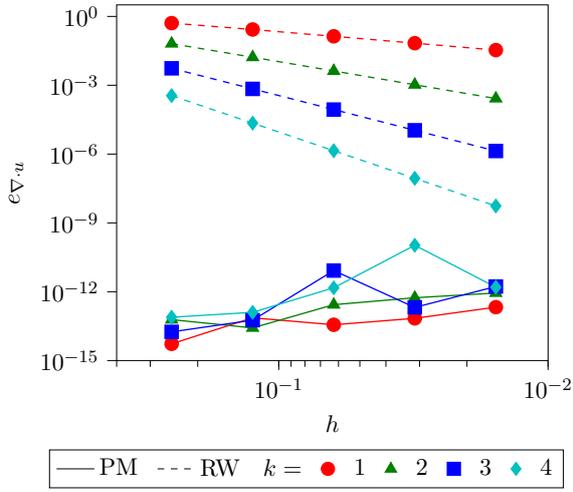
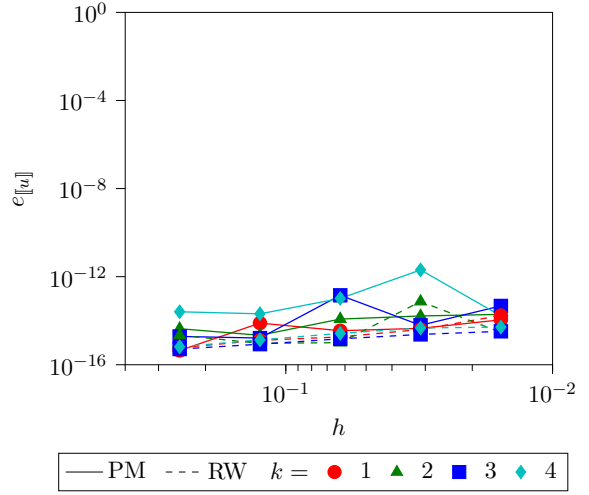
The computed velocity field for the present method is divergence-free to
machine precision, which is in agreement with the theory presented in
\cref{sec:ensuring_div_free}. By contrast, whilst the method from
\cite{Rhebergen2018a} gives a velocity field with continuous normal
component across element boundaries to machine precision, the
approximate velocity field is not exactly divergence-free.

\Cref{fig:square_2_press_rob} shows $e_u$ and $e_p$ as a function of $h$
for the present method and the scheme presented in \cite{Rhebergen2018a}
with $k=2$ and $\nu \in \cbr{1, 10^{-3}, 10^{-6}}$.
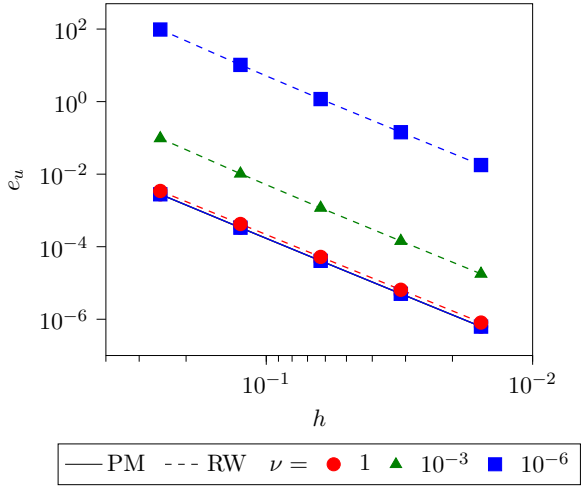
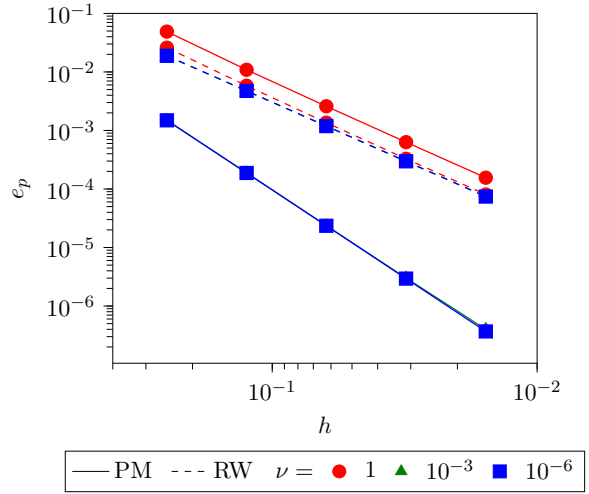
\begin{figure}
    \centering
    \begin{subfigure}[b]{0.49\textwidth}
        \centering
        \resizebox{\textwidth}{!}{
        \input{square2_press_rob_conv_u.tex}}
        \caption{Velocity field}
    \end{subfigure}
    \hfill
    \begin{subfigure}[b]{0.49\textwidth}
        \centering
        \resizebox{\textwidth}{!}{
        \input{square2_press_rob_conv_p.tex}}
        \caption{Pressure field}
    \end{subfigure}
    \caption{Velocity ($e_u$) and pressure ($e_p$) errors as a
    function of $h$ with $k=2$ for the present method (PM) and the
    method presented in \cite{Rhebergen2018a} (RW). The error in the
    velocity field is independent of the viscosity for the present
    method, in contrast to the method from~\cite{Rhebergen2018a}.}
    \label{fig:square_2_press_rob}
\end{figure}
For the present method, the error in the velocity does not depend on the
viscosity, which is in agreement with the pressure robust estimate given
in \cref{th:error_estimate}. By contrast, for the method from
\cite{Rhebergen2018a}, the error in the velocity increases by several
orders of magnitude as the viscosity is varied. This is because the
computed velocity field is not exactly divergence-free, and thus the
velocity error estimate contains the norm of the error in the pressure
field scaled by the reciprocal of the viscosity \cite{Rhebergen2020}. In
the case of the pressure field, for the present method, the error
reduces as the viscosity decreases from $1$ to $10^{-3}$, but remains
constant when the viscosity is further decreased to $10^{-6}$. This is
consistent with the pressure error estimate given in
\cref{th:error_estimate}, which contains the $H^k(\Omega)$-norm of the
exact pressure field and the $H^{k+1}(\Omega)$ norm of the exact
velocity field scaled by the viscosity. For large enough values of
$\nu$, the velocity term dominates and therefore decreasing the
viscosity reduces the error. However, when $\nu$ is small enough, the
pressure term dominates, so reducing $\nu$ further has little effect. In
fact, for small viscosities, the pressure converges at a rate $k + 1$
instead of $k$. This is due to two factors: firstly, since the pressure
term dominates, the rate of convergence is not limited by the
$\mathcal{O}\del[0]{h^k}$ velocity term. Secondly, for this problem, the
exact pressure field is sufficiently regular to be in $H^{k +
1}(\Omega)$, and since the pressure space contains all polynomials of
degree at most $k$, the approximate pressure field converges at the rate
$k + 1$ (see \cite{Arnold2002b} and \cite[Remark 6.23]{DiPietro2012} for
more information). For the method from \cite{Rhebergen2018a}, varying
the viscosity has little effect on the error in the pressure field.

\subsection{A hydrostatic problem}

We consider the hydrostatic problem from \cite{John2017}. Let a fluid of
unit viscosity occupy a unit square container with no slip between the
fluid and the fixed container walls. The fluid is subjected to the force
\begin{equation}
    f(x) =
    \begin{pmatrix}
        0 \\
        c(3x_2^2 - x_2 + 1)
    \end{pmatrix},
\end{equation}
where $c \in \mathbb{R}^+$ is a parameter. The exact solution is given by
\begin{equation}
    u(x) = 0 \textnormal{ and } p(x) = c\left(x_2^3 - \frac{x_2^2}{2}
    + x_2 - \frac{7}{12}\right).
\end{equation}
Note that the applied force is exactly balanced by the pressure gradient;
the fluid is in hydrostatic equilibrium and changing the parameter $c$ only
changes the pressure field.

The solution is computed using the present method and a Taylor--Hood
scheme with $k = 2$, $c = 10^4$, and $\nu = 1$ on a trapezium partition
of $\Omega \coloneqq (0, 1)^2$. Homogeneous Dirichlet boundary
conditions are applied on $\partial \Omega$ for the velocity field.
\Cref{fig:hydrostatic_vel} shows the magnitude of the approximate
velocity field computed using each method, and the errors are tabulated
in \cref{tab:hydrostatic}.
\begin{figure}
    \centering
    \begin{subfigure}[b]{0.49\textwidth}
        \centering
        \includegraphics[width=\textwidth]{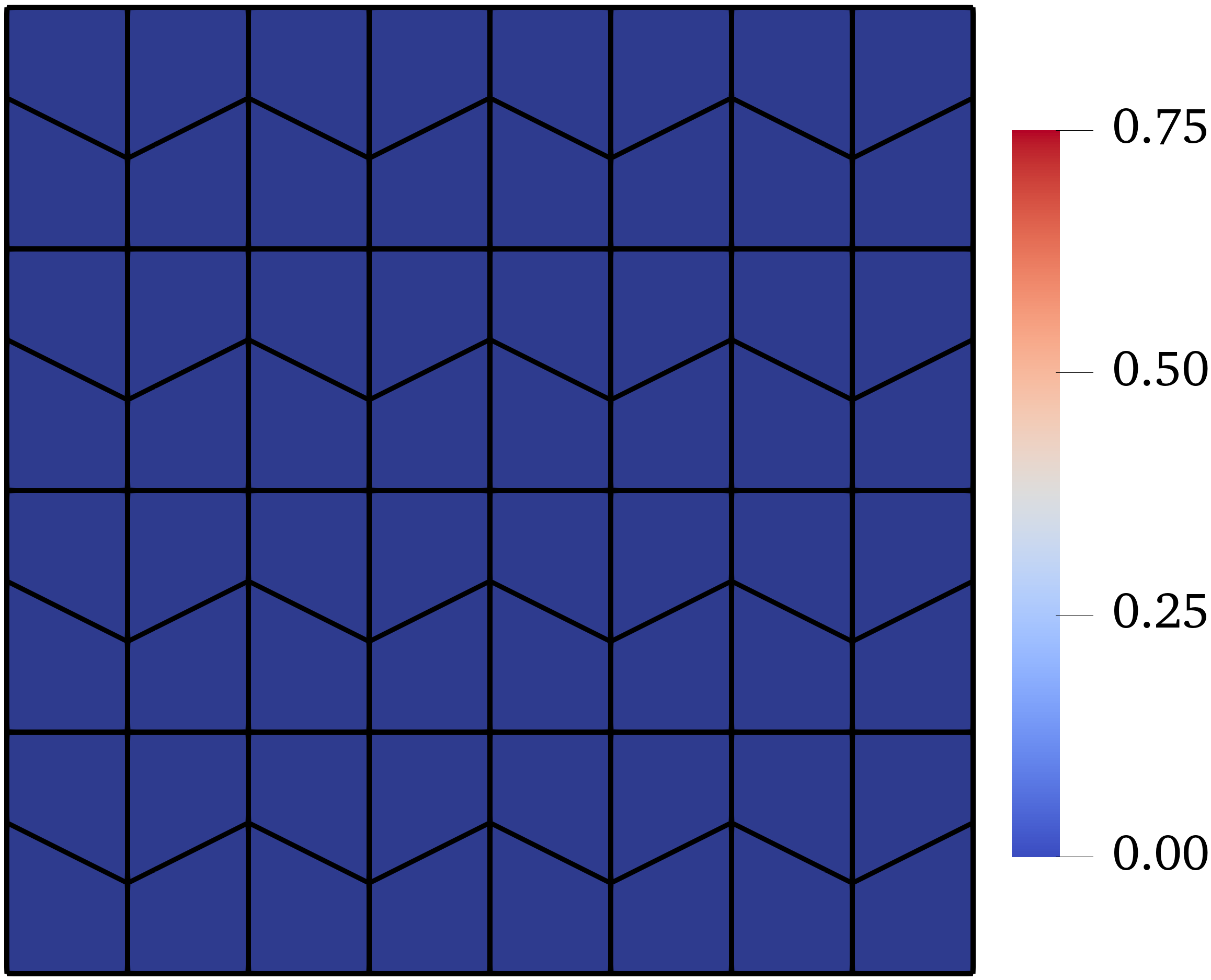}
        \caption{Present method}
    \end{subfigure}
    \hfill
    \begin{subfigure}[b]{0.49\textwidth}
        \centering
        \includegraphics[width=\textwidth]{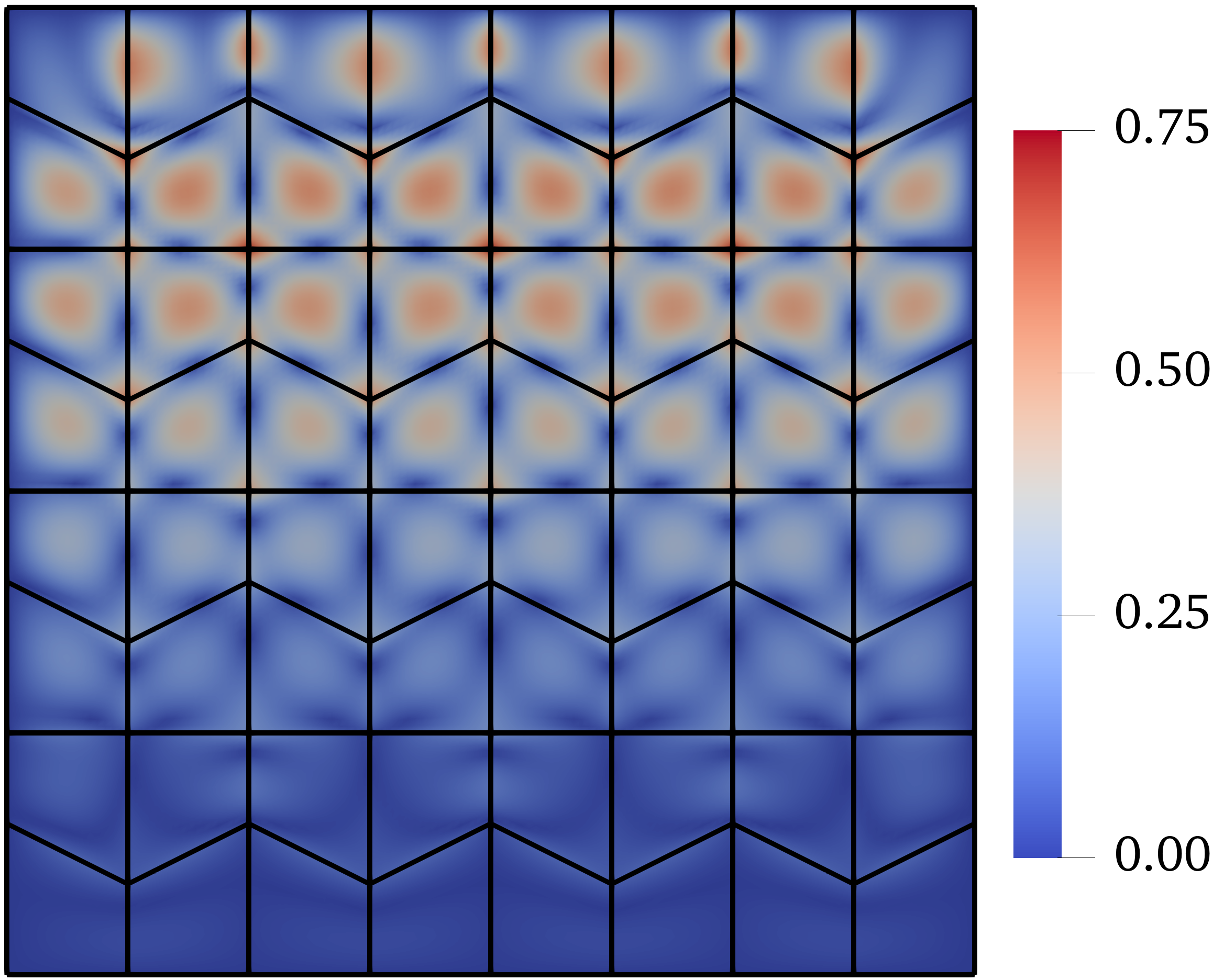}
        \caption{Taylor--Hood}
    \end{subfigure}
    \caption{Magnitude of the computed velocity field for the
    hydrostatic problem with $k = 2$ and $c = 10^4$ for the present
    method and a Taylor--Hood scheme.}
    \label{fig:hydrostatic_vel}
\end{figure}
\input{hydrostatic_table.tex}
The velocity field computed by the present method is exact to machine
precision. By contrast, spurious flow can be seen in the velocity field
computed by the Taylor--Hood scheme, which unlike the present method,
does not preserve the invariance property of the Stokes equations.

\subsection{A cylindrical bearing problem}

We now present an example involving a curved boundary, which we
approximate with curved elements. Whilst we showed in
\cref{sec:ensuring_div_free} that discretely divergence-free functions
are exactly divergence-free for the curved elements we use in this
example, our theoretical analysis used to obtain the pressure robust
error estimate in \cref{sec:press_rob_error_est} does not account for
the fact that $\Omega_h$ does not coincide exactly with $\Omega$, and we
have not proved \crefrange{ass:consistency}{ass:approx_props} for these
elements. Despite this, it will be seen experimentally that the
invariance property is preserved.

Consider a two-dimensional domain $\Omega \coloneqq \left\{x_1^2 + x_2^2
< r_{\rm{o}}^2\right\} \setminus \left\{ x_1^2 + (x_2 + e)^2 <
r_{\rm{i}}^2 \right\}$ bounded by inner and outer circles of radii
$r_{\rm{i}}$ and $r_{\rm{o}}$ respectively. Let $e$ denote the offset
between their centres in the $x_2$-direction. On the inner and outer
boundaries, the tangential component of the velocity is prescribed as
$u_{\rm{i}}$ and $u_{\rm{o}}$ respectively, and the normal component is
set to zero. The applied force is taken to be zero. This type of flow
has an analytical solution which can be found in \cite{Wannier1950}.

We take $r_{\rm{i}} = 0.7$, $r_{\rm{o}} = 1$, $e = 0.15$, $u_{\rm{i}} =
1$, and $u_{\rm{o}} = 0$, and use an unstructured mesh containing both
quadrilateral and triangular cells. Polynomial geometric mappings of
degree four are used to curve the cells on the boundary to ensure it is
represented with sufficient accuracy. The computed solution for $k = 3$
is presented in \cref{fig:wannier_sol}.
\begin{figure}
    \centering
    \begin{subfigure}[b]{0.49\textwidth}
        \centering
        \includegraphics[width=\textwidth]{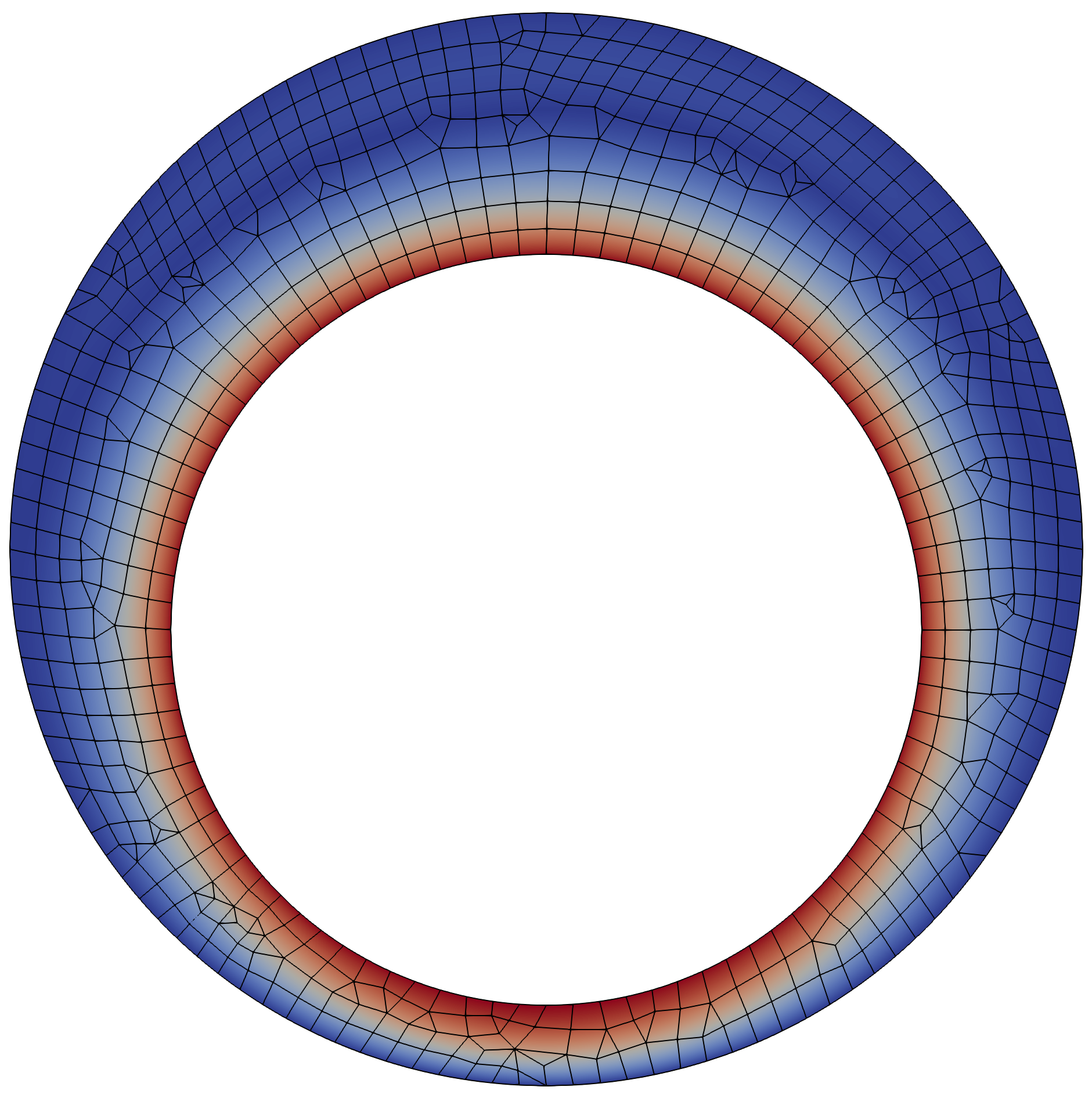}
        \caption{Velocity field magnitude}
    \end{subfigure}
    \hfill
    \begin{subfigure}[b]{0.49\textwidth}
        \centering
        \includegraphics[width=\textwidth]{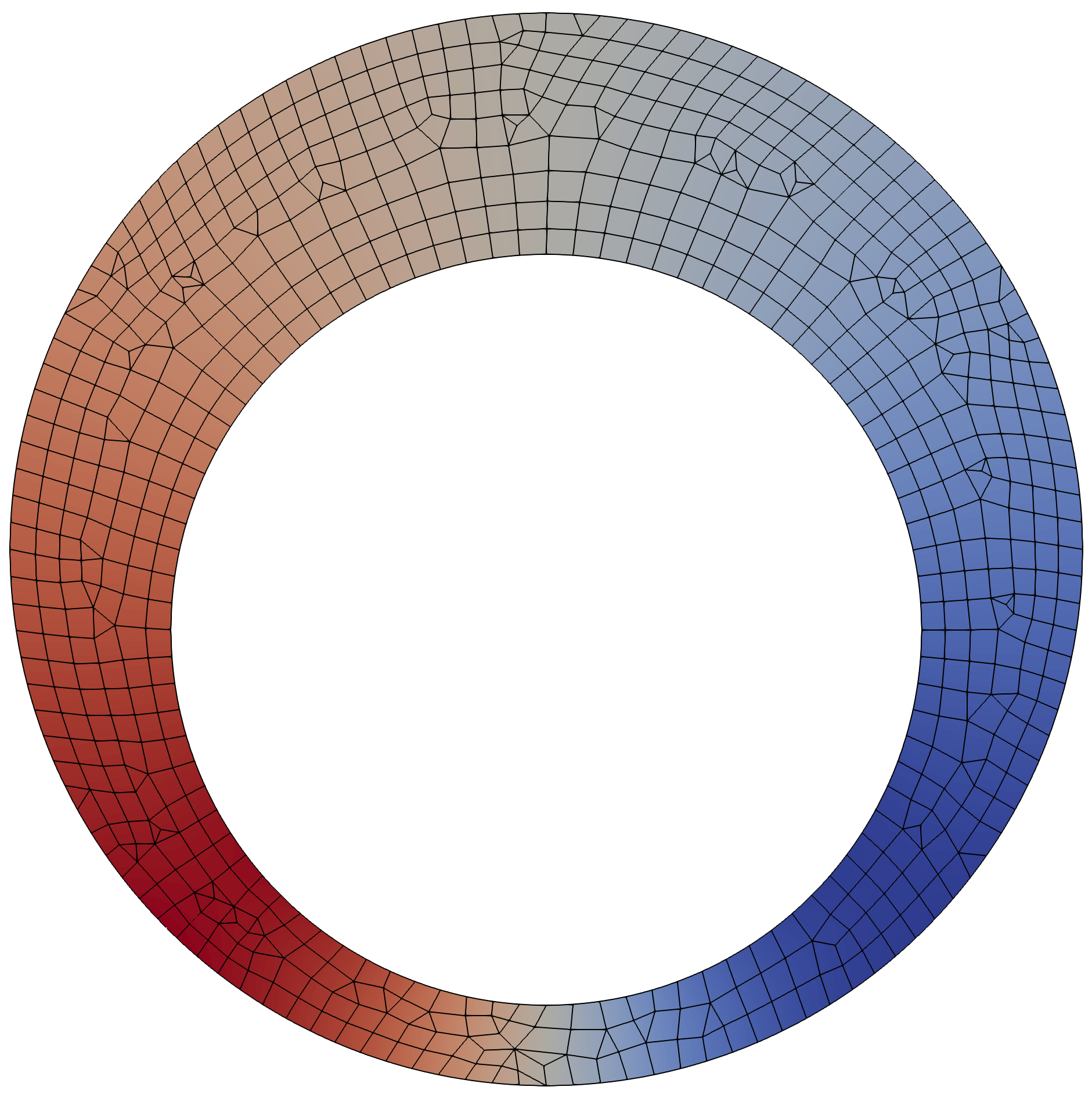}
        \caption{Pressure field}
    \end{subfigure}
    \caption{The computed velocity magnitude and pressure fields for
    the cylindrical bearing problem with $k = 3$.}
    \label{fig:wannier_sol}
\end{figure}
\Cref{fig:wannier_conv_u} shows that the computed velocity converges at
the rate $k + 1$.
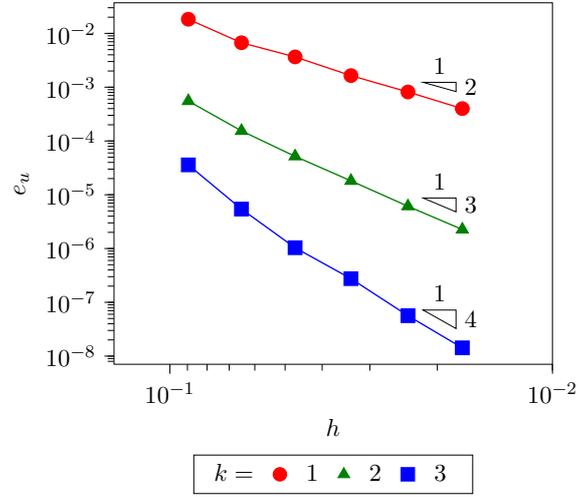
\begin{figure}
    \centering
    \resizebox{0.49\textwidth}{!}{
    \input{wannier_conv_l2_errors_u.tex}}
    \caption{The $L^2(\Omega)$-norm of the error in the velocity
    field as a function of $h$. The convergence rate of the velocity
    is $k + 1$.}
    \label{fig:wannier_conv_u}
\end{figure}
If curved cells had not been used on the boundary, values of $k$ larger
than $1$ would not have increased the rate of convergence of the
velocity field due to the poor geometric approximation.
\Cref{fig:wannier_conv_div_jump} shows $e_{\nabla \cdot u}$ and
$e_{\llbracket u \rrbracket}$ for the present method and method from
\cite{Rhebergen2018a} as a function of~$h$.
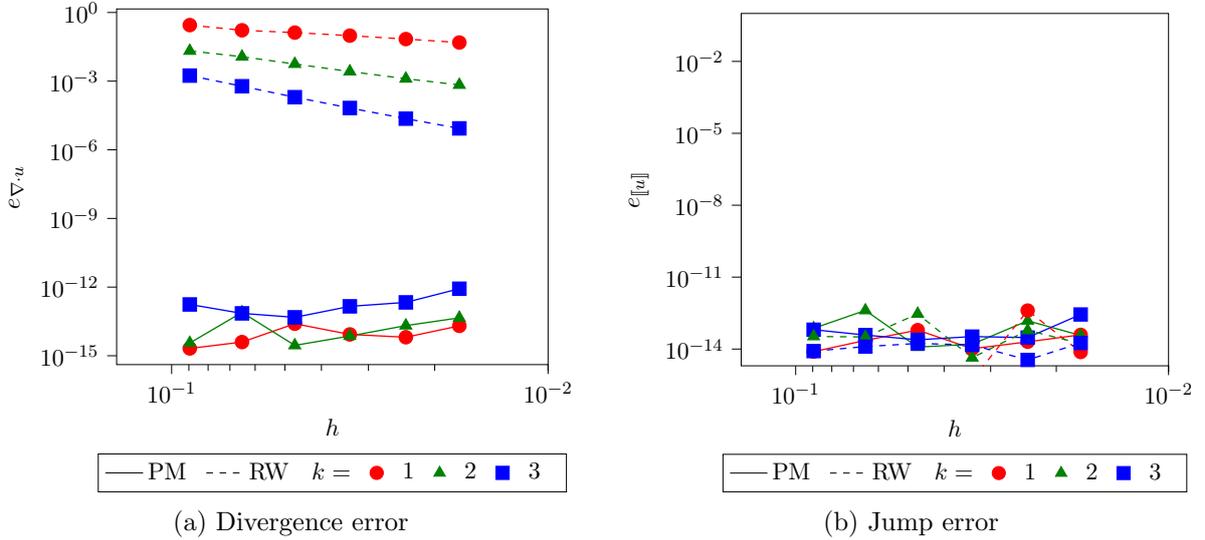
\begin{figure}
    \begin{subfigure}[b]{0.49\textwidth}
        \centering
        \resizebox{\textwidth}{!}{
        \input{wannier_conv_div_errors.tex}}
        \caption{Divergence error}
    \end{subfigure}
    \hfill
    \begin{subfigure}[b]{0.49\textwidth}
        \centering
        \resizebox{\textwidth}{!}{
        \input{wannier_conv_jump_errors.tex}}
        \caption{Jump error}
    \end{subfigure}
    \caption{The divergence and jump errors as a function of $h$ for the
    present method (PM) and the method from \cite{Rhebergen2018a} (RW).
    The present method conserves mass exactly, the method from
    \cite{Rhebergen2018a} does not.}
    \label{fig:wannier_conv_div_jump}
\end{figure}
Once again, the present method gives a velocity field that is exactly
divergence-free to machine precision despite the presence of curved
elements, which is in agreement with the theory presented in
\cref{sec:ensuring_div_free}. By contrast, the method from
\cite{Rhebergen2018a} produces a velocity field that is $H(\rm{div};
\Omega)$-conforming to machine precision but not exactly
divergence-free.

Finally, we demonstrate that, despite the presence of curved elements in
this test case, the present method preserves the invariance property of
the Stokes equations at the discrete level. Let the source function be
modified to
\begin{equation}
    f(x) = c \nabla \left( \sin(\pi x_2) \right),
\end{equation}
where $c \in \mathbb{R}^+$ is a parameter. Since $f$ is a gradient
field, it is irrotational and therefore is exactly balanced by the
pressure gradient in the continuous problem, leaving the velocity field
unchanged \cite{John2017}. The solution was computed for the present
method and the method from \cite{Rhebergen2018a} with $k = 2$ and $c$
ranging from $1$ to $10^6$. The results are presented in
\cref{fig:wannier_press_rob}.
\begin{figure}
    \centering
    \resizebox{0.49\textwidth}{!}{
    \input{wannier_press_rob_conv_u.tex}}
    \caption{The error in the velocity field as a function of $h$ for
    the present method (PM) and the method from \cite{Rhebergen2018a}
    (RW). The error in the velocity is independent of $c$ for the
    present method, in contrast to the method from
    \cite{Rhebergen2018a}.}
    \label{fig:wannier_press_rob}
\end{figure}
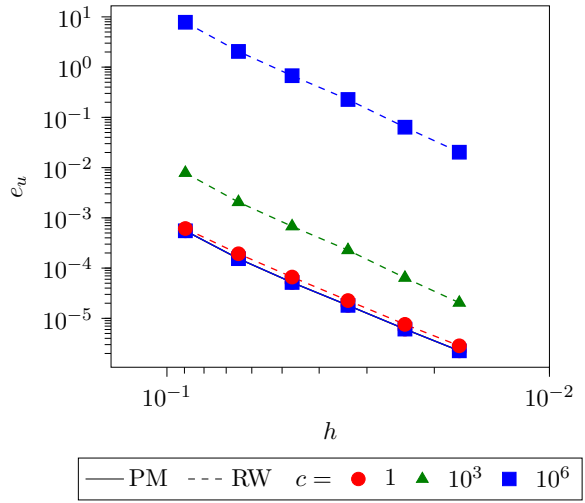
For the present method, the velocity field is unchanged as $c$ is
varied, which is in accordance with the physically correct behaviour. By
contrast, the method from \cite{Rhebergen2018a} does not enjoy this
property.

\section{Conclusions} \label{sec:conclusion}

We have developed and analysed a hybridized discontinuous Galerkin
method for the Stokes problem on non-affine cells that maintains an
exactly divergence-free velocity field. Elements are mapped using the
contravariant Piola transform, ensuring that key properties hold on both
the reference and physical (mapped) element. Optimal, pressure robust
error estimates for the velocity field are proved for flat-faced
quadrilateral cells. Our theoretical analysis is consistent with our
numerical experiments, which demonstrate that invariance properties are
preserved and the analytical error estimates achieved. The method
developed in this paper opens up the application of high-order, pressure
robust methods for problems with complex geometries, with future
possibilities being analysis of high aspect ratio tensor-product cell
meshes for boundary layers and the generalisation of the analysis
results to other non-affine cell types.

\subsection*{Acknowledgements}

Support from the Engineering and Physical Sciences Research Council for
JPD (EP/L015943/1) and GNW (EP/W026635/1, EP/W00755X/1, EP/S005072/1) is
gratefully acknowledged. SR gratefully acknowledges support from the
Natural Sciences and Engineering Research Council of Canada through the
Discovery Grant program (RGPIN-2023-03237). We further thank Aaron
Baier-Reinio for discussions on the proof of \cref{lem:stability_b_2}
resulting in an improved inf-sup constant.

\bibliographystyle{unsrtnat}
\bibliography{references}
\end{document}

%% file: square2_conv_l2_errors_u.tex
\begin{tikzpicture}

\definecolor{color0}{rgb}{0,0.75,0.75}

\begin{axis}[
legend style={at={(0.5, -0.25)}, anchor=north, legend columns=5},
log basis x={10},
log basis y={10},
tick align=outside,
tick pos=left,
x dir=reverse,
x grid style={white!69.0196078431373!black},
xlabel={\(\displaystyle h\)},
xmin=0.01, xmax=0.4,
xmode=log,
xtick style={color=black},
y grid style={white!69.0196078431373!black},
ylabel={\(\displaystyle e_u\)},
ymin=1e-12, ymax=0.2,
ymode=log,
ytick style={color=black}
]
\addlegendimage{white, only marks, mark=*}
\addlegendentry{$k = $ \;}
\addlegendimage{red, only marks, mark=*, mark size=3}
\addlegendentry{\; $1$ \;}
\addlegendimage{green!50!black, only marks, mark=triangle*, mark size=3}
\addlegendentry{\; $2$ \;}
\addlegendimage{blue, only marks, mark=square*, mark size=3}
\addlegendentry{\; $3$ \;}
\addlegendimage{color0, only marks, mark=diamond*, mark size=3}
\addlegendentry{\; $4$ \;}
\addplot [semithick, red, mark=*, mark size=3, mark options={solid}]
table {%
0.25 0.0439547318014267
0.125 0.011161637687467
0.0625 0.00280695263365612
0.03125 0.000702961551694355
0.015625 0.000175835233578192
};
\addplot [semithick, green!50!black, mark=triangle*, mark size=3, mark options={solid}]
table {%
0.25 0.00316509215910099
0.125 0.000381711044023056
0.0625 4.67841341382072e-05
0.03125 5.79222184995535e-06
0.015625 7.20609563106093e-07
};
\addplot [semithick, blue, mark=square*, mark size=3, mark options={solid}]
table {%
0.25 0.000197153724557929
0.125 1.22939877694143e-05
0.0625 7.65146286978468e-07
0.03125 4.76825035729661e-08
0.015625 2.97522998338422e-09
};
\addplot [semithick, color0, mark=diamond*, mark size=3, mark options={solid}]
table {%
0.25 1.00842824086361e-05
0.125 3.14131101336275e-07
0.0625 9.77873420762271e-09
0.03125 3.05274716244734e-10
0.015625 9.70934673681963e-12
};
\logLogSlopeTriangle{0.84}{0.1}{0.82}{2}{black};
\logLogSlopeTriangle{0.84}{0.1}{0.62}{3}{black};
\logLogSlopeTriangle{0.84}{0.1}{0.43}{4}{black};
\logLogSlopeTriangle{0.84}{0.1}{0.23}{5}{black};
\end{axis}

\end{tikzpicture}

%% file: square2_conv_l2_errors_p.tex
\begin{tikzpicture}

\definecolor{color0}{rgb}{0,0.75,0.75}

\begin{axis}[
legend style={at={(0.5, -0.25)}, anchor=north, legend columns=5},
log basis x={10},
log basis y={10},
tick align=outside,
tick pos=left,
x dir=reverse,
x grid style={white!69.0196078431373!black},
xlabel={\(\displaystyle h\)},
xmin=0.01, xmax=0.4,
xmode=log,
xtick style={color=black},
y grid style={white!69.0196078431373!black},
ylabel={\(\displaystyle e_p\)},
ymin=1e-9, ymax=1,
ymode=log,
ytick style={color=black}
]
\addlegendimage{white, only marks, mark=*}
\addlegendentry{$k = $ \;}
\addlegendimage{red, only marks, mark=*, mark size=3}
\addlegendentry{\; $1$ \;}
\addlegendimage{green!50!black, only marks, mark=triangle*, mark size=3}
\addlegendentry{\; $2$ \;}
\addlegendimage{blue, only marks, mark=square*, mark size=3}
\addlegendentry{\; $3$ \;}
\addlegendimage{color0, only marks, mark=diamond*, mark size=3}
\addlegendentry{\; $4$ \;}
\addplot [semithick, red, mark=*, mark size=3, mark options={solid}]
table {%
0.25 0.427037206314468
0.125 0.165527179510691
0.0625 0.0694845215782169
0.03125 0.0318217119397989
0.015625 0.0153224628877838
};
\addplot [semithick, green!50!black, mark=triangle*, mark size=3, mark options={solid}]
table {%
0.25 0.0684601000405118
0.125 0.0150614778409351
0.0625 0.00363250415426264
0.03125 0.000894773964943089
0.015625 0.000222112971525111
};
\addplot [semithick, blue, mark=square*, mark size=3, mark options={solid}]
table {%
0.25 0.0100525195302451
0.125 0.00108775046458195
0.0625 0.000127809991256094
0.03125 1.55963703715211e-05
0.015625 1.93114930180795e-06
};
\addplot [semithick, color0, mark=diamond*, mark size=3, mark options={solid}]
table {%
0.25 0.000568713587551805
0.125 3.34096494843463e-05
0.0625 2.06105507053671e-06
0.03125 1.29218894662298e-07
0.015625 9.50834785326159e-09
};
\logLogSlopeTriangle{0.84}{0.1}{0.87}{1}{black};
\logLogSlopeTriangle{0.84}{0.1}{0.67}{2}{black};
\logLogSlopeTriangle{0.84}{0.1}{0.47}{3}{black};
\logLogSlopeTriangle{0.84}{0.1}{0.24}{4}{black};
\end{axis}

\end{tikzpicture}

%% file: square2_conv_div_errors.tex
\begin{tikzpicture}

\definecolor{color0}{rgb}{0,0.75,0.75}

\begin{axis}[
legend style={at={(0.44, -0.25)}, anchor=north, legend columns=6},
log basis x={10},
log basis y={10},
tick align=outside,
tick pos=left,
x dir=reverse,
x grid style={white!69.0196078431373!black},
xlabel={\(\displaystyle h\)},
xmin=0.01, xmax=0.4,
xmode=log,
xtick style={color=black},
y grid style={white!69.0196078431373!black},
ylabel={\(\displaystyle e_{\nabla \cdot u}\)},
ymin=1e-15, ymax=3.17900394440214,
ymode=log,
ytick style={color=black}
]
\addlegendimage{black, line legend}
\addlegendentry{PM \;}
\addlegendimage{black, line legend, dashed}
\addlegendentry{RW \; $k = $ \;}
\addlegendimage{red, only marks, mark=*, mark size=3}
\addlegendentry{\; $1$ \;}
\addlegendimage{green!50!black, only marks, mark=triangle*, mark size=3}
\addlegendentry{\; $2$ \;}
\addlegendimage{blue, only marks, mark=square*, mark size=3}
\addlegendentry{\; $3$ \;}
\addlegendimage{color0, only marks, mark=diamond*, mark size=3}
\addlegendentry{\; $4$}
\addplot [semithick, red, mark=*, mark size=3, mark options={solid}]
table {%
0.25 5.35640285798828e-15
0.125 7.29772125538404e-14
0.0625 3.63351184302517e-14
0.03125 6.92576026255107e-14
0.015625 2.1265287513277e-13
};
\addplot [semithick, green!50!black, mark=triangle*, mark size=3, mark options={solid}]
table {%
0.25 6.077738987833e-14
0.125 2.65295279900785e-14
0.0625 2.72298695916976e-13
0.03125 5.45597921952871e-13
0.015625 8.79815363761102e-13
};
\addplot [semithick, blue, mark=square*, mark size=3, mark options={solid}]
table {%
0.25 1.78150710027407e-14
0.125 5.74504726903893e-14
0.0625 8.36670958712906e-12
0.03125 2.07317740593888e-13
0.015625 1.68571155634211e-12
};
\addplot [semithick, color0, mark=diamond*, mark size=3, mark options={solid}]
table {%
0.25 7.73347597867658e-14
0.125 1.26514883436291e-13
0.0625 1.48695890484373e-12
0.03125 1.05784658818933e-10
0.015625 1.57596899578016e-12
};
\addplot [semithick, red, dashed, mark=*, mark size=3, mark options={solid}]
table {%
0.25 0.510396359511402
0.125 0.268959216399347
0.0625 0.135941124697004
0.03125 0.0681194540084616
0.015625 0.034075092349399
};
\addplot [semithick, green!50!black, dashed, mark=triangle*, mark size=3, mark options={solid}]
table {%
0.25 0.065004747036219
0.125 0.0166545020005149
0.0625 0.00419685166136032
0.03125 0.00105230661783215
0.015625 0.000263397892547828
};
\addplot [semithick, blue, dashed, mark=square*, mark size=3, mark options={solid}]
table {%
0.25 0.0054652850119785
0.125 0.000693766742595781
0.0625 8.70634113998566e-05
0.03125 1.0894010593415e-05
0.015625 1.3621217938246e-06
};
\addplot [semithick, color0, dashed, mark=diamond*, mark size=3, mark options={solid}]
table {%
0.25 0.00035258730106345
0.125 2.2314617586566e-05
0.0625 1.39914588583748e-06
0.03125 8.75210432948192e-08
0.015625 5.47136011690383e-09
};
\end{axis}

\end{tikzpicture}

%% file: square2_conv_jump_errors.tex
\begin{tikzpicture}

\definecolor{color0}{rgb}{0,0.75,0.75}

\begin{axis}[
legend style={at={(0.44, -0.25)}, anchor=north, legend columns=6},
log basis x={10},
log basis y={10},
tick align=outside,
tick pos=left,
x dir=reverse,
x grid style={white!69.0196078431373!black},
xlabel={\(\displaystyle h\)},
xmin=0.01, xmax=0.4,
xmode=log,
xtick style={color=black},
y grid style={white!69.0196078431373!black},
ylabel={\(\displaystyle e_{\llbracket u \rrbracket}\)},
ymin=1e-16, ymax=1,
ymode=log,
ytick style={color=black}
]
\addlegendimage{black, line legend}
\addlegendentry{PM \;}
\addlegendimage{black, line legend, dashed}
\addlegendentry{RW \; $k = $ \;}
\addlegendimage{red, only marks, mark=*, mark size=3}
\addlegendentry{\; $1$ \;}
\addlegendimage{green!50!black, only marks, mark=triangle*, mark size=3}
\addlegendentry{\; $2$ \;}
\addlegendimage{blue, only marks, mark=square*, mark size=3}
\addlegendentry{\; $3$ \;}
\addlegendimage{color0, only marks, mark=diamond*, mark size=3}
\addlegendentry{\; $4$}
\addplot [semithick, red, mark=*, mark size=3, mark options={solid}]
table {%
0.25 4.41893585041917e-16
0.125 7.79117572437499e-15
0.0625 3.5177855872832e-15
0.03125 4.492156586025e-15
0.015625 1.09626732133525e-14
};
\addplot [semithick, green!50!black, mark=triangle*, mark size=3, mark options={solid}]
table {%
0.25 4.32410451915502e-15
0.125 2.18105135162513e-15
0.0625 1.18561648897305e-14
0.03125 1.64020720270878e-14
0.015625 1.9411913921882e-14
};
\addplot [semithick, blue, mark=square*, mark size=3, mark options={solid}]
table {%
0.25 1.92697027624972e-15
0.125 1.63024824712137e-15
0.0625 1.42117308593428e-13
0.03125 6.18535368934408e-15
0.015625 4.69106097118512e-14
};
\addplot [semithick, color0, mark=diamond*, mark size=3, mark options={solid}]
table {%
0.25 2.54368109986053e-14
0.125 2.05327488590569e-14
0.0625 1.02513302927631e-13
0.03125 2.02215308164438e-12
0.015625 1.65047635160199e-14
};
\addplot [semithick, red, dashed, mark=*, mark size=3, mark options={solid}]
table {%
0.25 5.94099529202108e-16
0.125 1.28351999317882e-15
0.0625 1.83577473494226e-15
0.03125 3.99290131932392e-15
0.015625 1.72074008930544e-14
};
\addplot [semithick, green!50!black, dashed, mark=triangle*, mark size=3, mark options={solid}]
table {%
0.25 2.17407072155183e-15
0.125 9.37946254718417e-16
0.0625 1.00887823130235e-15
0.03125 7.42108087493859e-14
0.015625 2.94914878567891e-15
};
\addplot [semithick, blue, dashed, mark=square*, mark size=3, mark options={solid}]
table {%
0.25 5.16489217901429e-16
0.125 8.58985327080527e-16
0.0625 1.50323097152034e-15
0.03125 2.38352966803558e-15
0.015625 3.28551856360102e-15
};
\addplot [semithick, color0, dashed, mark=diamond*, mark size=3, mark options={solid}]
table {%
0.25 6.55007659329193e-16
0.125 1.34877206930437e-15
0.0625 2.66150100707619e-15
0.03125 4.68089177755934e-15
0.015625 5.12818275051251e-15
};
\end{axis}

\end{tikzpicture}

%% file: square2_press_rob_conv_u.tex
\begin{tikzpicture}

\begin{axis}[
legend style={at={(0.5, -0.25)}, anchor=north, legend columns=5},
log basis x={10},
log basis y={10},
tick align=outside,
tick pos=left,
x dir=reverse,
x grid style={white!69.0196078431373!black},
xlabel={\(\displaystyle h\)},
xmin=0.01, xmax=0.4,
xmode=log,
xtick style={color=black},
y grid style={white!69.0196078431373!black},
ylabel={\(\displaystyle e_u\)},
ymin=1e-07, ymax=5e2,
ymode=log,
ytick style={color=black}
]
\addlegendimage{black, line legend}
\addlegendentry{PM \;}
\addlegendimage{black, line legend, dashed}
\addlegendentry{RW \; $\nu = $ \;}
\addlegendimage{red, only marks, mark=*, mark size=3}
\addlegendentry{\; $1$ \;}
\addlegendimage{green!50!black, only marks, mark=triangle*, mark size=3}
\addlegendentry{\; $10^{-3}$ \;}
\addlegendimage{blue, only marks, mark=square*, mark size=3}
\addlegendentry{\; $10^{-6}$}
\addplot [semithick, red, mark=*, mark size=3, mark options={solid}]
table {%
0.25 0.00278724032049497
0.125 0.000334202224078561
0.0625 4.08050180648948e-05
0.03125 5.04122116290091e-06
0.015625 6.2648657905258e-07
};
\addplot [semithick, green!50!black, mark=triangle*, mark size=3, mark options={solid}]
table {%
0.25 0.00278724032049527
0.125 0.000334202224078471
0.0625 4.08050180646313e-05
0.03125 5.04122116280112e-06
0.015625 6.26486578997253e-07
};
\addplot [semithick, blue, mark=square*, mark size=3, mark options={solid}]
table {%
0.25 0.00278724032590327
0.125 0.00033420221098342
0.0625 4.08050179588413e-05
0.03125 5.04122115673546e-06
0.015625 6.26486568575368e-07
};
\addplot [semithick, red, dashed, mark=*, mark size=3, mark options={solid}]
table {%
0.25 0.00343908309191402
0.125 0.00042275700626943
0.0625 5.22839505798613e-05
0.03125 6.49963353352475e-06
0.015625 8.10214310622748e-07
};
\addplot [semithick, green!50!black, dashed, mark=triangle*, mark size=3, mark options={solid}]
table {%
0.25 0.0972164739357568
0.125 0.0102884001651091
0.0625 0.00117645455147224
0.03125 0.000143385546607913
0.015625 1.7849839397655e-05
};
\addplot [semithick, blue, dashed, mark=square*, mark size=3, mark options={solid}]
table {%
0.25 97.155673859173
0.125 10.2797159650293
0.0625 1.175292765653
0.03125 0.143238229193515
0.015625 0.0178314508269547
};
\end{axis}

\end{tikzpicture}

%% file: square2_press_rob_conv_p.tex
\begin{tikzpicture}

\begin{axis}[
legend style={at={(0.5, -0.25)}, anchor=north, legend columns=5},
log basis x={10},
log basis y={10},
tick align=outside,
tick pos=left,
x dir=reverse,
x grid style={white!69.0196078431373!black},
xlabel={\(\displaystyle h\)},
xmin=0.01, xmax=0.4,
xmode=log,
xtick style={color=black},
y grid style={white!69.0196078431373!black},
ylabel={\(\displaystyle e_p\)},
ymin=1-07, ymax=0.1,
ymode=log,
ytick style={color=black}
]
\addlegendimage{black, line legend}
\addlegendentry{PM \;}
\addlegendimage{black, line legend, dashed}
\addlegendentry{RW \; $\nu = $ \;}
\addlegendimage{red, only marks, mark=*, mark size=3}
\addlegendentry{\; $1$ \;}
\addlegendimage{green!50!black, only marks, mark=triangle*, mark size=3}
\addlegendentry{\; $10^{-3}$ \;}
\addlegendimage{blue, only marks, mark=square*, mark size=3}
\addlegendentry{\; $10^{-6}$}
\addplot [semithick, red, mark=*, mark size=3, mark options={solid}]
table {%
0.25 0.0488910973650589
0.125 0.0108933999819412
0.0625 0.00258948666478103
0.03125 0.000631457396590214
0.015625 0.000155893776610266
};
\addplot [semithick, green!50!black, mark=triangle*, mark size=3, mark options={solid}]
table {%
0.25 0.0014911985235232
0.125 0.000188233266139664
0.0625 2.3682541658488e-05
0.03125 3.01111408516601e-06
0.015625 3.99722463662621e-07
};
\addplot [semithick, blue, mark=square*, mark size=3, mark options={solid}]
table {%
0.25 0.00149039757139806
0.125 0.000187917885658097
0.0625 2.35405585974746e-05
0.03125 2.94416009361973e-06
0.015625 3.68069746324576e-07
};
\addplot [semithick, red, dashed, mark=*, mark size=3, mark options={solid}]
table {%
0.25 0.0259541202728887
0.125 0.00581477900404244
0.0625 0.00136027897822479
0.03125 0.000327918954154924
0.015625 8.04403917363363e-05
};
\addplot [semithick, green!50!black, dashed, mark=triangle*, mark size=3, mark options={solid}]
table {%
0.25 0.0188659550855638
0.125 0.00474453076531723
0.0625 0.00118799820484939
0.03125 0.00029709360893791
0.015625 7.42750136640536e-05
};
\addplot [semithick, blue, dashed, mark=square*, mark size=3, mark options={solid}]
table {%
0.25 0.0188659466658425
0.125 0.00474452957435822
0.0625 0.00118799802008048
0.03125 0.000297093576513079
0.015625 7.42750072426525e-05
};
\end{axis}

\end{tikzpicture}

%% file: hydrostatic_table.tex
\begin{table}
\centering
\caption{Computed errors for the hydrostatic problem using the present method and a
Taylor-Hood scheme.}
\label{tab:hydrostatic}
\begin{tabular}{|c|c|c|c|c|}
\hline
Method & $e_u$ & $e_p$ & $e_{\nabla \cdot u}$ & $e_{\llbracket u \rrbracket}$ \\ \hline
Present & $\num{1.10E-13}$ & $\num{0.447}$ & $\num{6.54E-13}$ & $\num{4.56E-14}$\\ \hline
Taylor-Hood & $\num{0.226}$ & $\num{19.3}$ & $\num{8.37}$ & $\num{1.57E-16}$\\ \hline
\end{tabular}
\end{table}

%% file: wannier_conv_l2_errors_u.tex
\begin{tikzpicture}

\begin{axis}[
legend style={at={(0.5, -0.25)}, anchor=north, legend columns=5},
log basis x={10},
log basis y={10},
tick align=outside,
tick pos=left,
x dir=reverse,
x grid style={white!69.0196078431373!black},
xlabel={\(\displaystyle h\)},
xmin=0.01, xmax=0.14,
xmode=log,
xtick style={color=black},
y grid style={white!69.0196078431373!black},
ylabel={\(\displaystyle e_u\)},
ymin=7.05073273755758e-09, ymax=0.0370480579851576,
ymode=log,
ytick style={color=black}
]
\addlegendimage{white, only marks, mark=*}
\addlegendentry{$k = $ \;}
\addlegendimage{red, only marks, mark=*, mark size=3}
\addlegendentry{\; $1$ \;}
\addlegendimage{green!50!black, only marks, mark=triangle*, mark size=3}
\addlegendentry{\; $2$ \;}
\addlegendimage{blue, only marks, mark=square*, mark size=3}
\addlegendentry{\; $3$ \;}
\addplot [semithick, red, mark=*, mark size=3, mark options={solid}]
table {%
0.089504532099007 0.018335246669856
0.0650190200165734 0.00670486516582394
0.0470750592981567 0.00365545127497897
0.0336259915320607 0.00164306360802903
0.0238657539209129 0.000815666647291653
0.0172044586123659 0.000399835959472358
};
\addplot [semithick, green!50!black, mark=triangle*, mark size=3, mark options={solid}]
table {%
0.089504532099007 0.000552480779682936
0.0650190200165734 0.00015408107888004
0.0470750592981567 5.17950876163116e-05
0.0336259915320607 1.81283496889619e-05
0.0238657539209129 6.11257850776583e-06
0.0172044586123659 2.25689139723126e-06
};
\addplot [semithick, blue, mark=square*, mark size=3, mark options={solid}]
table {%
0.089504532099007 3.59289687957562e-05
0.0650190200165734 5.41521240512503e-06
0.0470750592981567 1.03502641263977e-06
0.0336259915320607 2.75463933202373e-07
0.0238657539209129 5.63885525291281e-08
0.0172044586123659 1.42466561809792e-08
};
\logLogSlopeTriangle{0.78}{0.075}{0.78}{2}{black};
\logLogSlopeTriangle{0.78}{0.075}{0.46}{3}{black};
\logLogSlopeTriangle{0.78}{0.075}{0.15}{4}{black};
\end{axis}

\end{tikzpicture}

%% file: wannier_conv_div_errors.tex
\begin{tikzpicture}

\begin{axis}[
legend style={at={(0.5, -0.25)}, anchor=north, legend columns=5},
log basis x={10},
log basis y={10},
tick align=outside,
tick pos=left,
x dir=reverse,
x grid style={white!69.0196078431373!black},
xlabel={\(\displaystyle h\)},
xmin=0.01, xmax=0.14,
xmode=log,
xtick style={color=black},
y grid style={white!69.0196078431373!black},
ylabel={\(\displaystyle e_{\nabla \cdot u}\)},
ymin=4.1035249681987e-16, ymax=1.42922529226321,
ymode=log,
ytick style={color=black}
]
\addlegendimage{black, line legend}
\addlegendentry{PM \;}
\addlegendimage{black, line legend, dashed}
\addlegendentry{RW \; $k = $ \;}
\addlegendimage{red, only marks, mark=*, mark size=3}
\addlegendentry{\; $1$ \;}
\addlegendimage{green!50!black, only marks, mark=triangle*, mark size=3}
\addlegendentry{\; $2$ \;}
\addlegendimage{blue, only marks, mark=square*, mark size=3}
\addlegendentry{\; $3$ \;}
\addplot [semithick, red, mark=*, mark size=3, mark options={solid}]
table {%
0.089504532099007 2.08741626273388e-15
0.0650190200165734 3.96143284153601e-15
0.0470750592981567 2.50798879690285e-14
0.0336259915320607 8.55788147653836e-15
0.0238657539209129 6.43862660362055e-15
0.0172044586123659 2.0205964190133e-14
};
\addplot [semithick, green!50!black, mark=triangle*, mark size=3, mark options={solid}]
table {%
0.089504532099007 3.59943591909253e-15
0.0650190200165734 7.84195546769528e-14
0.0470750592981567 2.83720220397926e-15
0.0336259915320607 7.28915710720079e-15
0.0238657539209129 2.07419095132628e-14
0.0172044586123659 4.53039156473309e-14
};
\addplot [semithick, blue, mark=square*, mark size=3, mark options={solid}]
table {%
0.089504532099007 1.74395274589428e-13
0.0650190200165734 7.11053960370938e-14
0.0470750592981567 4.76536502940793e-14
0.0336259915320607 1.44029668466045e-13
0.0238657539209129 2.14542165773416e-13
0.0172044586123659 8.48047178153246e-13
};
\addplot [semithick, red, dashed, mark=*, mark size=3, mark options={solid}]
table {%
0.089504532099007 0.280962727783963
0.0650190200165734 0.166809112229147
0.0470750592981567 0.131887031689049
0.0336259915320607 0.0972617547772568
0.0238657539209129 0.0688201503310766
0.0172044586123659 0.0486307830595266
};
\addplot [semithick, green!50!black, dashed, mark=triangle*, mark size=3, mark options={solid}]
table {%
0.089504532099007 0.020981020208005
0.0650190200165734 0.0116621028968149
0.0470750592981567 0.00560136260747251
0.0336259915320607 0.00263515139446603
0.0238657539209129 0.00127463664591813
0.0172044586123659 0.000682628811992352
};
\addplot [semithick, blue, dashed, mark=square*, mark size=3, mark options={solid}]
table {%
0.089504532099007 0.00174335260976226
0.0650190200165734 0.000597762655845057
0.0470750592981567 0.000198853399177322
0.0336259915320607 6.70530403964251e-05
0.0238657539209129 2.29032232026655e-05
0.0172044586123659 8.61318033834117e-06
};
\end{axis}

\end{tikzpicture}

%% file: wannier_conv_jump_errors.tex
\begin{tikzpicture}

\begin{axis}[
legend style={at={(0.5, -0.25)}, anchor=north, legend columns=5},
log basis x={10},
log basis y={10},
tick align=outside,
tick pos=left,
x dir=reverse,
x grid style={white!69.0196078431373!black},
xlabel={\(\displaystyle h\)},
xmin=0.01, xmax=0.14,
xmode=log,
xtick style={color=black},
y grid style={white!69.0196078431373!black},
ylabel={\(\displaystyle e_{\llbracket u \rrbracket}\)},
ymin=2e-15, ymax=1,
ymode=log,
ytick style={color=black}
]
\addlegendimage{black, line legend}
\addlegendentry{PM \;}
\addlegendimage{black, line legend, dashed}
\addlegendentry{RW \; $k = $ \;}
\addlegendimage{red, only marks, mark=*, mark size=3}
\addlegendentry{\; $1$ \;}
\addlegendimage{green!50!black, only marks, mark=triangle*, mark size=3}
\addlegendentry{\; $2$ \;}
\addlegendimage{blue, only marks, mark=square*, mark size=3}
\addlegendentry{\; $3$ \;}
\addplot [semithick, red, mark=*, mark size=3, mark options={solid}]
table {%
0.089504532099007 7.78276748747794e-15
0.0650190200165734 2.29459159550089e-14
0.0470750592981567 6.17107299003714e-14
0.0336259915320607 1.03383757546775e-14
0.0238657539209129 1.99836406613366e-14
0.0172044586123659 3.91760771102905e-14
};
\addplot [semithick, green!50!black, mark=triangle*, mark size=3, mark options={solid}]
table {%
0.089504532099007 7.32018236882611e-14
0.0650190200165734 4.08491903838833e-13
0.0470750592981567 1.20502998010196e-14
0.0336259915320607 1.58957789762669e-14
0.0238657539209129 1.50976714003435e-13
0.0172044586123659 3.6997580097925e-14
};
\addplot [semithick, blue, mark=square*, mark size=3, mark options={solid}]
table {%
0.089504532099007 6.4110735344958e-14
0.0650190200165734 3.82914773472001e-14
0.0470750592981567 2.43047942974773e-14
0.0336259915320607 3.31010250008373e-14
0.0238657539209129 3.0107710875623e-14
0.0172044586123659 2.76758985764058e-13
};
\addplot [semithick, red, dashed, mark=*, mark size=3, mark options={solid}]
table {%
0.089504532099007 3.34457049254153e-16
0.0650190200165734 3.48972461813443e-16
0.0470750592981567 4.48113520419446e-16
0.0336259915320607 5.16017649590693e-16
0.0238657539209129 4.04153809033743e-13
0.0172044586123659 7.48253137974094e-15
};
\addplot [semithick, green!50!black, dashed, mark=triangle*, mark size=3, mark options={solid}]
table {%
0.089504532099007 3.36271503990098e-14
0.0650190200165734 3.17286050269835e-14
0.0470750592981567 2.9350692508445e-13
0.0336259915320607 4.34741666911224e-15
0.0238657539209129 6.00623439115477e-14
0.0172044586123659 1.71385270373076e-14
};
\addplot [semithick, blue, dashed, mark=square*, mark size=3, mark options={solid}]
table {%
0.089504532099007 8.19331623082884e-15
0.0650190200165734 1.26700001964335e-14
0.0470750592981567 1.70826088406305e-14
0.0336259915320607 1.48649845981374e-14
0.0238657539209129 3.51592972351656e-15
0.0172044586123659 1.78797004385566e-14
};
\end{axis}

\end{tikzpicture}

%% file: wannier_press_rob_conv_u.tex
\begin{tikzpicture}

\begin{axis}[
legend style={at={(0.5, -0.25)}, anchor=north, legend columns=5},
log basis x={10},
log basis y={10},
tick align=outside,
tick pos=left,
x dir=reverse,
x grid style={white!69.0196078431373!black},
xlabel={\(\displaystyle h\)},
xmin=0.01, xmax=0.14,
xmode=log,
xtick style={color=black},
y grid style={white!69.0196078431373!black},
ylabel={\(\displaystyle e_u\)},
ymin=1.06303299482219e-06, ymax=16.5864292366663,
ymode=log,
ytick style={color=black}
]
\addlegendimage{black, line legend}
\addlegendentry{PM \;}
\addlegendimage{black, line legend, dashed}
\addlegendentry{RW \; $c = $ \;}
\addlegendimage{red, only marks, mark=*, mark size=3}
\addlegendentry{\; $1$ \;}
\addlegendimage{green!50!black, only marks, mark=triangle*, mark size=3}
\addlegendentry{\; $10^{3}$ \;}
\addlegendimage{blue, only marks, mark=square*, mark size=3}
\addlegendentry{\; $10^{6}$}
\addplot [semithick, red, mark=*, mark size=3, mark options={solid}]
table {%
0.089504532099007 0.000552480779684595
0.0650190200165734 0.000154081078878173
0.0470750592981567 5.17950876164871e-05
0.0336259915320607 1.81283496890377e-05
0.0238657539209129 6.11257850781714e-06
0.0172044586123659 2.25689139713563e-06
};
\addplot [semithick, green!50!black, mark=triangle*, mark size=3, mark options={solid}]
table {%
0.089504532099007 0.000552480780863935
0.0650190200165734 0.000154081078959987
0.0470750592981567 5.17950876184421e-05
0.0336259915320607 1.81283496886123e-05
0.0238657539209129 6.11257850774187e-06
0.0172044586123659 2.25689139723167e-06
};
\addplot [semithick, blue, mark=square*, mark size=3, mark options={solid}]
table {%
0.089504532099007 0.000552481962763478
0.0650190200165734 0.000154081162803288
0.0470750592981567 5.17950897800648e-05
0.0336259915320607 1.81283494590692e-05
0.0238657539209129 6.11257851685209e-06
0.0172044586123659 2.25689139509503e-06
};
\addplot [semithick, red, dashed, mark=*, mark size=3, mark options={solid}]
table {%
0.089504532099007 0.000613217590498487
0.0650190200165734 0.000192896712053565
0.0470750592981567 6.6352128285796e-05
0.0336259915320607 2.26020534574948e-05
0.0238657539209129 7.59248400691962e-06
0.0172044586123659 2.83838732634471e-06
};
\addplot [semithick, green!50!black, dashed, mark=triangle*, mark size=3, mark options={solid}]
table {%
0.089504532099007 0.00784283441893469
0.0650190200165734 0.00205971054837052
0.0470750592981567 0.000677715333488168
0.0336259915320607 0.000227658887872689
0.0238657539209129 6.41568003995157e-05
0.0172044586123659 2.04460744522394e-05
};
\addplot [semithick, blue, dashed, mark=square*, mark size=3, mark options={solid}]
table {%
0.089504532099007 7.81248117795112
0.0650190200165734 2.05279500113836
0.0470750592981567 0.675073214000505
0.0336259915320607 0.226908326981389
0.0238657539209129 0.0636968922592336
0.0172044586123659 0.020250920453047
};
\end{axis}

\end{tikzpicture}